\documentclass[12pt,twoside]{article}
\usepackage{amsmath,amsthm,amssymb,amsfonts}
\usepackage{graphicx}
\usepackage{hyperref}

\theoremstyle{cplain}
\newtheorem{Thrm}{Theorem}[section]
	\newtheorem{Prop}[Thrm]{Proposition}
	\newtheorem{Cor}[Thrm]{Corollary}
	\newtheorem{Lemma}[Thrm]{Lemma}
{
\theoremstyle{cremark}

}
{
\theoremstyle{uremark}

}
{
\theoremstyle{cdefinition}
\newtheorem{Def}[Thrm]{Definition}
}

\newcommand{\ehf}{\text{EHF}}
\newcommand{\dehf}{\text{DEHF}}
\newcommand{\val}{\text{VAL}}
\newcommand{\pat}{\text{PAT}}
\newcommand{\bb}{\text{BB}}
\newcommand{\bc}{\text{BC}}

\begin{document}
\title{Some examples of asymptotic \\ combinatorial behavior, zero-one and
convergence results on random \\  hypergraphs}
\author{Nicolau C. Saldanha, M\'arcio Telles}
\maketitle



\abstract{%

This is an extended version of the thesis presented to the Programa de P\'os-Gradua\c c\~ao em
Matem\'atica of the Departamento de Matem\'atica, PUC-Rio, in September 2013,  incorporating some suggestions from the examining commission.

Random graphs (and more generally hypergraphs) have been extensively studied,
including their first order logic. In this work we focus on certain
specific aspects
of this vast theory. We consider the binomial model $G^{d+1}(n,p)$ of the random
$(d+1)$-uniform hypergraph on $n$ vertices, where each edge is present,
independently of one another,
with probability $p=p(n)$. We are particularly interested in the range
$p(n) \sim C\log(n)/n^d$,
after the double jump and near connectivity. We prove several
zero-one, and, more generally,
convergence results and obtain combinatorial applications of some of them. 
}

                                                \section{Introduction}
    
  This is an extended version of the thesis presented to the Programa de P\'os-Gradua\c c\~ao em
Matem\'atica of the Departamento de Matem\'atica, PUC-Rio, in September 2014,  incorporating some suggestions from the examining commission. Among such extensions of the original text,
there is a discussion of convergence laws for the random hypergraph in the range
    $$p\sim\frac{\lambda}{n^d},$$
called the \emph{Double Jump}.

The original version can be found in the following address:

\url{http://www.dbd.puc-rio.br/pergamum/tesesabertas/0821520_2013_completo.pdf}

    It has now been more than fifty years since Erd\H os and R\'enyi laid the foundations for the
    study of random graphs on their seminal paper \emph{On the evolution of random graphs}
    \cite{erdos2}, where they considered the binomial random graph model $G(n,p)$. This consists of
     a graph on $n$ vertices
    where each of the potential ${n\choose 2}$ edges is present with probability $p$, all these events
    being independent of each other. Many interesting asymptotic questions arise when $n$ tends to $\infty$
    and we let $p$ depend on $n$. 
    
   Among other results, they showed that many properties of graphs exhibit a \emph{threshold
   behavior}, meaning that the probability that the property holds on $G(n,p)$ turns from near $0$ to near $1$ in a 
   narrow range of the edge probability $p$. More concretely, given a property $P$ of graphs, in
   many cases there is a \emph{threshold} function $p:\mathbb{N}\to[0,1]$ such that, as $n\to\infty$, the probability that $G(n,\tilde{p})$
   satisfies $P$ tends to $0$ for all $\tilde{p}\ll p$ and tends to $1$ for all $\tilde{p}\gg p$. Erd\H os
   and R\'enyi showed, for example, that the threshold for connectivity is $p=\frac{\log n}{n}$.
   They also showed that there is a profound change in the component structure of $G(n,p)$ for 
   $p$ around $\frac{1}{n}$, where one of its many connected components suddenly becomes
   dramatically larger than all others, a phenomenon mainly understood today as a phase transition. The range of $p$ where this occur is called the \emph{Double Jump} and 
   has received enormous attention from researchers since then.
   
   The above threshold behaviors suggest that one could expect to describe some convergence
   results, where the probabilities of all properties in a certain class converge to 
   known values as $n\to\infty$. Among the first convergence results there are the \emph{zero-one laws}, where all properties of graphs expressible by a first order formula (called \emph{elementary properties}) converge to $0$ or 
   to $1$. This happens, for example, if $p$ is independent of $n$. Many other instances of zero-one laws for random graphs were obtained by Joel Spencer in the book \emph{The Strange Logic of Random Graphs} \cite{spencer}. There he shows that zero-one laws hold if $p$ lies between a number of ``critical" functions. More concretely, if $p$ satisfies one of the following conditions

   \begin{enumerate}
   
      \item $p\ll n^{-2}$
      \item $n^{-\frac{1+l}{l}}\ll p\ll n^{-\frac{2+l}{1+l}}$ for some $l\in\mathbb{N}$
      \item $n^{-1-\epsilon}\ll p\ll n^{-1}$ for all $\epsilon>0$
      \item $n^{-1}\ll p\ll (\log n)n^{-1}$
      \item $(\log n)n^{-1}\ll p\ll n^{-1+\epsilon}$ for all $\epsilon>0$
   
    \end{enumerate}
    then a zero-one law holds.
        
    Note that clause $2$ is, in fact, a \emph{scheme} of clauses. Note also that $1$ can be viewed
    as a special case of $2$. There are functions $p\ll n^{-1+\epsilon}$ not considered by 
    any of the above conditions. Such ``gaps" occur an infinite number of times near the critic
    functions in the scheme $2$ and two more times: one between clauses $3$ and $4$ and the 
    other between clauses $4$ and $5$. Spencer shows that, for some functions $p$ conveniently near that ``critic" functions in $2$ and the critic functions $n^{-1}$ and $\frac{\log n}{n}$, corresponding to the gaps $3-4$ and $4-5$ respectively, the probabilities 
    of all elementary properties converge to constants $c\in[0,1]$ as
    $n\to\infty$. This situation, more general than that of a zero-one law, is called a \emph{convergence law}. 
    
    One sees immediately that the possibility that an edge probability function oscillates infinitely often between two 
    different values can be an obstruction to getting convergence laws. With this difficulty in mind,
    we consider the edge probability functions $p:\mathbb{N}\to[0,1]$ that belong to Hardy's class
    of \emph{logarithmo-exponential functions}. This class is entirely made of eventually monotone 
    functions, avoiding the above mentioned problem, and has the additional convenience of being closed by elementary algebraic operations and compositions that can involve
    logarithms and exponentials. All thresholds of natural properties of graphs seem to belong
    in Hardy's class.
    
    Generally speaking, our work implies that, once one restricts the edge probabilities
    to functions in Hardy's class, there are no further ``gaps": all logarithmo-exponential edge 
    probabilities $p\ll n^{-1+\epsilon}$ are convergence laws. The arguments in Spencer's book
    are sufficient for getting most of these convergence laws, except for those in the window
    $p\sim C\frac{\log n}{n}$ , $C>0$, where just the value $C=1$ is discussed.
    
    One of our main interests lies in the completion of the discussion of the convergence laws in the window $p\sim C\frac{\log n}{n}$ for other values of $C$ and generalizations of the beautiful arguments in Spencer's book
    to random uniform hypergraphs.  We will see that this window hides an infinite collection of zero-one and convergence laws and that those can be presented in a simple organized fashion. We also get simple axiomatizations of the almost sure theories and describe all elementary
    types of the countable models of these theories.
    
    Convergence laws have a deep connection with some elementary concepts of logic. More 
    precisely, zero-one laws occur when the class $T_p$ of almost sure elementary properties is 
    \emph{complete} and convergence laws occur when this ``almost sure theory" is, in a sense,
    almost complete. Some everyday results in first-order logic imply that when all countably infinite models of $T_p$   are \emph{elementarily equivalent} (that is, satisfy the same elementary properties), $T_p$ is complete. This is obviously the case if there is, apart from
    isomorphism, only one such model: in this case, we say that $T_p$ is $\aleph_0$-categorical.
    We will face situations where the countable models of
    $T_p$ are, indeed, unique up to isomorphism. In other cases, the almost sure theory is still complete but the countable models are not unique: in the
    instances of the latter situation, the countable models are not far from being uniquely determined and, in particular, lend themselves to an exhaustive characterization.  
    Finally, there are cases where $T_p$ is not complete but we still have convergence laws: in
    these cases, the almost sure theories are not far from being complete, and we still manage to 
    classify their countable models.
    
    Along the way, we describe some combinatorial aspects of the component structure of the
    random hypergraph in the window $p\sim C\frac{\log n}{n}$, including some estimates of
    the size of the complement of its largest connected component. As a consequence, we get some elementary 
    approximations of non-elementary events that work for probability edge functions in 
    Hardy's class.     
    The phase transition occurring in the
    Double Jump $p\sim Cn^{-1}$ has recently been seen to hold, in this more general context of 
    random $(d+1)$-uniform hypergraphs, in the window $p\sim Cn^{-d}$ by Schmidt-Pruzan and
    Shamir in \cite{pruzan}. We do not discuss convergence laws in this window for $d>1$ as
    Spencer successfully does for $d=1$, although this seems an interesting question worthy of
    clarification in the future.

                                                \section{Preliminaires}

          \subsection{The model $G^{d+1}(n,p)$}

We consider the binomial model $G^{d+1}(n,p)$ of random $(d+1)$-uniform  hypergraphs on $n$ vertices with probability $p\in[0,1]$, that is to say, the finite probability space on the set of all hypergraphs on $n$ labelled vertices
where each edge is a set of cardinality $d+1$ and if $H$ is such a hypergraph with $k$ edges then one has
 
                      $$\mathbb{P}[H]=p^k(1-p)^{{n\choose d+1}-k}.$$
                       
Another useful characterization of the same probability space is to insist that each of the potential ${n\choose d+1}$ edges be present in $G^{d+1}(n,p)$ with probability $p$, each of these events being independent of each other.
We will, more often, prefer the latter because it is more convenient in applications.

In the literature, the reader will find that the notations $G^{d+1}(n,p)$,  $H^{d+1}(n,p)$ and even some others also stand for $G^{d+1}(n,p)$. Our choice reflects our mere personal taste.

Our interest lies on the asymptotic behavior of $G^{d+1}(n,p)$ when $n\to\infty$ and $p=p(n)$ is a function of $n$. More specifically, a \emph{property} of hypergraphs is a class of hypergraphs closed by hypergraph isomorphism.
Each property $P$ gives rise to a sequence

                        $$\mathbb{P}(P)=\mathbb{P}[P](n,p):=\mathbb{P}[G^{d+1}(n,p(n))\models P]$$
and it is the asymptotic behavior  of these sequences we shall be interested in: when they converge; if so, what the limits are and so on.

A property $P$ is said to hold \emph{asymptoticaly almost surely} (or simply \emph{almost surely}) if $\mathbb{P}[P]\to 1$. In this case we say simply that \emph{$P$ holds a.a.s.}. A property $P$ is said to hold \emph{almost never} if its negation $\lnot P$ holds almost surely. Very often, it is the case that a property $P$
turns from holding almost never to holding almost surely in a narrow range of the edge probability $p$.

\begin{Def}

We say $\tilde{p}:\mathbb{N}\to[0,1]$ is a \emph{threshold function} (or simply \emph{a threshold}) for $P$ if both the following conditions hold:

         \begin{enumerate}

   \item If $p\gg\tilde{p}$ then $P$ holds a.a.s. in $G^{d+1}(n,p)$.

  \item If $p\ll\tilde{p}$ then $\lnot P$ holds a.a.s. in $G^{d+1}(n,p)$.

         \end{enumerate}

\end{Def}

Above and in all that follows, for eventually positive functions $f,g:\mathbb{N}\to\mathbb{R}$, both  expressions $f\ll g$ and $f=o(g)$ mean $\lim_{n\to\infty}\frac{f(n)}{g(n)}=0$.

Note that thresholds, when they exist, are not uniquely defined. For example, if $\tilde{p}$ meets the requirements of being a threshold for $P$ then all functions $c\cdot \tilde{p}$, with $c\in\mathbb{R}^*_+$, also do. Therefore, strictly speaking, it would be more correct to talk about ``representatives" of the threshold. However, as distinctions of this nature usually play no role in what follows, we shall not bother the reader with them.

As far as thresholds are concerned, the following is a generalization of a classical result of Erd\H os, R\'enyi and Bollob\'as, stated and proved by Vantsyan in \cite{vants}.

\begin{Thrm}       \label{vantsyan}

Fix a finite $(d+1)$-uniform hypergraph $H$ and let 

       $$\rho:=\max\left\{\frac{|E(\tilde{H})|}{|V(\tilde{H})|}\mid\tilde{H}\subseteq H, E(\tilde{H})>0\right\}.$$

Then the function $p(n)=n^{-\frac{1}{\rho}}$ is a threshold for the property of containment of $H$ as a sub-hypergraph.

\end{Thrm}

That is to say, $n^{-\frac{1}{\rho}}$ is a threshold for the appearance of small sub-hypergraphs with maximal \emph{density} $\rho$. We will need this later.

Among similar results, we will see that $p(n)=\frac{\log n}{n^d}$ is a threshold for $G^{d+1}(n,p)$ being connected, apart from getting thresholds for other properties.

                \subsection{Logarithmico-Exponential Functions}

Generally speaking, our results relate to convergence. That is, showing that for all functions $p$ and properties $P$ in certain specified classes, the limit

                          $$\lim_ {n\to\infty}\mathbb{P}(P)(n,p)$$
exists. Moreover, one can usually get a nice description of these limits.

One obvious obstruction to getting such results is the possibility that the function $p$ can oscillate between two different values, so that the corresponding probabilities also do. This would, obviously, rule out convergence.

To overcome this difficulty, one can restrict the possible functions $p$ to a class entirely made of eventually monotone functions. One natural such choice is Hardy's class of \emph{logarithmico-exponential functions},
or \emph{$L$-functions} for short, consisting of the eventually defined real-valued functions defined by a finite combination of the ordinary algebraic symbols and the functional symbols $\log(\ldots)$ and $\exp(\ldots)$
on the variable $n$. To avoid trivialities such as

              $$\frac{e^{\sqrt{-n^2}}-e^{-\sqrt{-n^2}}}{2}$$
we require that, in all ``stages of construction", the functions take only real values.

By induction on the complexity of $L$-functions, one can easily show that this class meets our requirement and even more. We state the following and refer the interested reader to Hardy's book \emph{Orders of Infinity} \cite{hardy} for a proof.

\begin{Thrm}

Any L-function is eventually continuous, of constant
sign, and monotonic, and, as $n\to +\infty$, converges to a definite limit or tends to $\pm\infty$. In particular, if $f$ and $g$ are eventually positive $L$-functions, exactly one of the following relations holds.

\begin{enumerate}

\item $f\ll g$
\item$f\gg g$
\item $f\sim c\cdot g$, for some constant $c\in\mathbb{R}.$

\end{enumerate}
 
\end{Thrm}

Thresholds of natural properties of graphs and hypergraphs appear to have representatives that are $L$-functions. (Stating and deciding a formal counterpart of that claim seems to be an interesting problem)
This makes the choice of $L$-functions in the context of random hypergraphs a rather natural one.

            \subsection{First Order Logic of Hypergraphs}

Having narrowed the class of possible edge probability functions, we now turn to a similar procedure on the class of properties of hypergraphs.

The \emph{first order logic of $(d+1)$-uniform hypergraphs} $\mathcal{FO}$ is the relational logic with language $\{\sigma\}$, where $\sigma$ is a $(d+1)$-ary predicate.
The semantics of $\mathcal{FO}$ is given by quantification over vertices and giving the formula $\sigma(x_0,x_1,\ldots,x_d)$ the interpretation ``$\{x_0,x_1,\ldots,x_d\}$ is an edge".

We say a property $P$ of $(d+1)$-uniform hypergraphs is \emph{elementary} if it can be represented by a formula in $\mathcal{FO}$. In this case, we write
$P\in\mathcal{FO}$ and, when no possibility of confusion arises, make no notational distinction between $P$ and the first order formula defining it.

A \emph{first order theory} is simply a subclass $\mathcal{C}\subseteq\mathcal{FO}$, that is, a class of elementary properties.

Fix a first order theory $\mathcal{C}\subseteq\mathcal{FO}$ and a property $P\in\mathcal{FO}$. We say \emph{$P$ is a semantic consequence of $\mathcal{C}$}, and write $\mathcal{C}\models P$, if $P$ is satisfied in all hypergraphs that satisfy all
of $\mathcal{C}$, that is to say, if $H\models\mathcal{C}$ implies $H\models P$. One can define a deductive system in which all derivations are finite sequences of formulae in $\mathcal{FO}$, giving rise to the concept of 
$P$ being a \emph{syntatic consequence of $\mathcal{C}$}, meaning that some $\mathcal{FO}$-formula defining $P$ is the last term of a derivation that only uses formulae in $\mathcal{C}$ as axioms. 

One piece of information in G\"odel's Completeness Theorem is the fact that one can pick such a deductive system in a suitable fashion so as to make the concepts of semantic and syntatic consequences identical. As derivations are finite sequences of formulae, the following \emph{Compactness Result} is obvious.

\begin{Prop}

If $P$ is a semantic consequence of $\mathcal{C}$ then it is a semantic consequence of a finite subclass of $\mathcal{C}$.

In particular, if every finite subclass of $\mathcal{C}$ is consistent then $\mathcal{C}$ is consistent.

\end{Prop}

The ``in particular" part comes from substitution of $P$ by any contradictory property. A careful analysis of the argument on the proof of G\"odel's Completeness Theorem shows the \emph{Downward L\"owenhein-Skolem Theorem},
that if $\mathcal{C}$ is a consistent theory (that is, satisfied by some hypergraph) then there is a hypergraph on a countable number of vertices satisfying $\mathcal{C}$. 

In spite of our particular interest in elementary properties, our interest is by no means exclusive. Rather we will, at times, discuss relations among elementary properties and
provably non-elementary ones. As a matter of example and also for future reference, we define the events $D_l$.

In what follows, recall that the \emph{incidence graph} $G(H)$ of a hypergraph $H$ is a bipartite graph with $V(H)$ on one side and $E(H)$ on the other and such that, for all $v\in V(H)$ and $e\in E(H)$, there is an edge connecting
$v$ and $e$ in $G(H)$ if, and only if, $v\in e$ in $H$. We say a hypergraph is Berge-acyclic is its incidence graph has no cycles. From now on we shall refer to Berge-acyclic hypergraphs simply as \emph{acyclic hypergraphs}.

\begin{Def}

A \emph{butterfly} is a connected acyclic uniform hypergraph. The \emph{order} of a finite butterfly is its number of edges.

\end{Def}

Fix $l\in\mathbb{N}$.

\begin{Def}

Let a hypergraph satisfy $D_l$ if the complement of a connected component of maximal size is a disjoint union of finite butterflies, all of them of order less then $l$. 

\end{Def}

So $D_1$ is the event that the hypergraph is a union of a component and some isolated points. For convenience we adopt the convention that $D_0$ is the event of being connected.

In the case $d=1$ of graphs, it is a well known result of Erd\H os and R\'enyi that the property $D_0$ of being connected, in spite of not being an elementary one, is asymptotically equivalent to the absence of isolated vertices,
obviously an elementary property. The events $D_l$ are generalizations of $D_0$ to other values of $l$ and $d$. As one should naturally expect, these generalizations give rise to concepts that are still non-elementary.
We shall see that, in analogy with Erd\H os and R\'enyi's result, the events $D_l$ are also asymptotically elementary.

The proof that $D_l\notin\mathcal{FO}$ exemplifies a nice use of Compactness.

\begin{Prop}

For all $l\in\mathbb{N}$, $D_l\notin\mathcal{FO}$.

\end{Prop}

\begin{proof}

 Fix $l$ and suppose, for a contradiction, that $D_l\in\mathcal{FO}$.
We use compactness.

 Let $\{A,B\}$ be a cut in a finite hypergraph. The \emph{norm} of $\{A,B\}$ is the number $\min\{|A|,|B|\}$.  Call a cut \emph{bad} if none of the two sides of the cut is a disjoint union of butterflies of order less then $l$. So $D_l$ is the event that there are no bad cuts.

For each $m\in\mathbb{N}$, let $E_m$ be the event that all bad cuts have norm at least $m$. By explicitly enumerating and excluding all bad cuts of order $<m$, one sees that $E_m\in\mathcal{FO}$.

Consider the theory $T=\{E_0,E_1,E_2,\ldots\}\cup\{\lnot D_l\}$. As there are hypergraphs that are a disjoint union of two large butterflies, one sees at once that every finite sub-theory of $T$ is consistent.
On the other hand, it is obvious that T is itself inconsistent, in contradiction with compactness.

\end{proof}

An analysis of the above argument shows that, although $D_l$ is not elementary, it is the class of
hypergraphs satisfying all properties in a first order theory, namely the theory 
     $T=\{E_0,E_1,\ldots\}$. We say a property $P$ is \emph{axiomatizable} if there is a first order 
  theory $T$ such that, for all hypergraphs $H$, one has $H\models P$ if, and only if $H\models\sigma$ for
  all $\sigma\in T$. Of course, if there is such a finite $T$, $P$ is elementary.  
  So $D_l$ is axiomatizable but not elementary. 
  
Some properties are beyond even the expressible power of first order theories. Further
insight on the proof of the above proposition shows that that is the case of the negations 
             $\lnot D_l$. 

\begin{Prop}

For all $l\in\mathbb{N}$, $\lnot D_l$ is not axiomatizable.

\end{Prop}

\begin{proof}

For a contradiction, suppose the theory $T$ axiomatizes $D_l$. Then the theory

                    $$T\cup\{E_0,E_1,\ldots\}$$
 is inconsistent. But, as we have seen above, every finite subtheory of $T\cup\{E_0,E_1,\ldots\}$
 is consistent, in contradiction with compactness.

\end{proof}
              
Still, as the reader can easily verify, if $P$ is \emph{any} property, then the property
  of \emph{being finite and satisfy $P$} is axiomatizable.

        \subsection{Zero-One Laws and Complete Theories}

The above observations will be useful in obtaining the following convergence results involving all properties in $\mathcal{FO}$.

\begin{Def}

We say a function $p: \mathbb{N}\to [0,1]$ is a \emph{zero-one law} if, for all $P\in\mathcal{FO}$, one has

      $$\lim_{n\to\infty}\mathbb{P}(P)(n,p)\in\{0,1\}.$$

\end{Def}

Above we mean that for every $P\in\mathcal{FO}$ the limit exists and is either zero or one. 

There is a close connection between zero-one laws and the concept of completeness. We say a theory $\mathcal{C}$ is \emph{complete} if, for
every $P\in\mathcal{FO}$, exactly one of $\mathcal{C}\models P$ or $\mathcal{C}\models\lnot P$ holds. 

Given $p:\mathbb{N}\to[0,1]$, the \emph{almost sure theory} of $p$ is defined by

         $$T_p:=\{P\in\mathcal{FO} | \mathbb{P}(P)(n,p)\to 1\}.$$

So $T_p$ is the class of elementary properties of $G^{d+1}(n,p)$ that hold almost surely. Note that, as a contradiction never holds, $T_p$ is always consistent. Moreover as, for every $m\in\mathbb{N}$, the property of having at least 
$m$ vertices is elementary and holds almost surely, $T_p$ has no finite models.

The connection between completeness and zero-one laws is given by the following.

\begin{Thrm}

The function $p$ is a zero-one law if, and only if, $T_p$ is complete.

\end{Thrm}

\begin{proof}

Suppose $T_p$ is complete and fix $P\in\mathcal{FO}$. As $T_p$ is complete, either $P$ or $\lnot P$ is a semantic consequence of $T_p$. If $T_p\models P$, by compactness, there is a finite set
$\{P_1,P_2\ldots,P_k\}\subseteq T_p$ such that $\{P_1,P_2\ldots,P_k\}\models P$. Therefore

          $$\mathbb{P}(P_1\land P_2\cdots\land P_k)\leq\mathbb{P}(P).$$

As $\mathbb{P}(P_1\land P_2\cdots\land P_k)\to 1$ we have also $\mathbb{P}(P)\to 1$. Similarly, if $T_p\models\lnot P$ one has $\mathbb{P}(\lnot P)\to 1$, so that $\mathbb{P}(P)\to 0$. As $P$ was arbitrary, 
$p$ is a zero-one law.

Conversely, if $p$ is a zero-one law then, for any $P\in\mathcal{FO}$, we have either $P\in T_p$ or $\lnot P\in T_p$. One cannot have both, as $T_p$ is consistent. So $T_p$ is complete.

\end{proof}

As $T_p$ is consistent and has no finite models, G\"odel's Completeness Theorem and L\"owenhein-Skolem give that the requirement of $T_p$ being complete is equivalent to asking that all
countable models of $T_p$ satisfy exactly the same first-order properties, a situation described in Logic by saying that all countable models are \emph{elementarily equivalent}. One obvious sufficient condition is that $T_p$ be \emph{$\aleph_0$-categorical}, that is, that $T_p$ has, apart from isomorphism, a unique countable model. We shall see several examples where $T_p$ is $\aleph_0$-categorical and other examples where the countable models of $T_p$ are elementarily equivalent but not necessarily isomorphic.

We summarize the above observations in the following corollary, more suitable for our applications.

\begin{Cor}

A function $p$ is a zero-one law if, and only if, all models of the almost sure theory $T_p$ are elementarily equivalent. 
In particular, if $T_p$ is $\aleph_0$-categorical, then $p$ is a zero-one law.

\end{Cor}

Uses of the above result require the ability to recognize when any two models $H_1$ and $H_2$  of $T_p$ are elementarily
equivalent. This is, usually, a simple matter in case $H_1$ and $H_2$ are isomorphic.
It is convenient to have at hand an instrument suitable to detecting when two structures of a first-order theory are elementarily equivalent regardless of being isomorphic.

Next, we briefly discuss the definition and some results on the Ehrenfeucht Game, which is a classic example of such an instrument.

           \subsection{The Ehrenfeucht Game}

This game has two players, called Spoiler and Duplicator, and two uniform hypergraphs $H_1$ and $H_2$ conventionally on disjoint sets of vertices. These
hypergraphs are known to both players. The game has a certain number $k$ of rounds which is again known to both players.

In each round, Spoiler selects one vertex not previously selected in either hypergraph and then Duplicator selects another vertex not previously selected in the other hypergraph.
At the end of the $k$-th round, the vertices $x_1,\ldots,x_k$ have been chosen on $H_1$ and $y_1,\ldots,y_k$ on $H_2$. Duplicator wins if, for all
$\{i_0,i_1,\ldots,i_d\}\subseteq\{1,2\ldots,k\}$, $\{x_{i_0},\ldots,x_{i_d}\}$ is an edge in $H_1$ if and only if the corresponding $\{y_{i_0},\ldots,y_{i_d}\}$ is an edge in $H_2$.
Spoiler wins if Duplicator does not. We denote the above described game by $\ehf\left(H_1,H_2;k\right)$.

As a technical point, the above description of the game works only if $k\leq\min\{\left|H_1\right|,\left|H_2\right|\}$. If that is not the case, we adopt the convention that Duplicator 
wins the game if, and only if, $H_1$ and $H_2$ are isomorphic.

The connection of the Ehrenfeucht Game to first order logic is a classic in logic and was given for the first time by R. Fra\"iss\'e in his PhD thesis in the more general context of purely relational structures with finite predicate symbols.  A proof in the particular case of
graphs can be found in Joel Spencer's book \emph{The Strange Logic of Random Graphs} \cite{spencer}, whose argument applies, \emph{mutatis mutandis} to uniform hypergraphs.

\begin{Prop}

A necessary and sufficient condition for the hypergraphs $H_1$ and $H_2$ to be elementarily equivalent is that, for all $k\in\mathbb{N}$, Duplicator has a winning strategy for the game $\ehf(H_1,H_2;k)$.

\end{Prop}

Now it is easy to see the connection of the game to zero-one laws.

\begin{Cor}

If for all countable models $H_1$ and $H_2$ of $T_p$ and all $k\in\mathbb{N}$ Duplicator has a winning strategy for $\ehf(H_1,H_2;k)$ then $p$ is a zero-one law.

\end{Cor}

              \subsection{Some winning strategies for Duplicator}

Now we describe some situations when there is a winning strategy for Duplicator without $H_1$ and $H_2$ being necessarily isomorphic. All propositions we state below are analogous to propositions in Spencer's book 
\emph{The Strange Logic of Random Graphs}. There, all results are stated and proved for graphs but, again, all arguments apply, \emph{mutatis mutandis}, to the case of uniform hypergraphs.

In what follows, if $H$ is a $(d+1)$-uniform hypergraph and $x\in H$, the \emph{$a$-neighborhood} of $x$ is the restriction of $H$ to the set of vertices at distance at most $a$ from $x$. If $x_1,\ldots,x_u\in H$, the \emph{$a$-picture} of $x_1,\ldots,x_u$ is the union of the $a$-neighborhoods of the $x_i$. Let $x_1,\ldots,x_u\in H_1$ and $y_1,\ldots,y_u\in H_2$. Their $a$-pictures are called \emph{the same} if there is an isomorphism between the $a$-pictures that sends
$x_i$ to $y_i$ for all $i\in\{1,\ldots,u\}$. Also, $\rho(x,y)$ is the distance from $x$ to $y$.

\begin{Prop}

Set $a=\frac{3^k-1}{2}$. Suppose $H_1$ and $H_2$ have vertex subsets $S_1\subseteq H_1$ and $S_2\subseteq H_2$ with the following properties:

  \begin{enumerate}

     \item The restrictions of $H_1$ to $S_1$ and $H_2$ to $S_2$ are isomorphic and this isomorphism can be extended to one between the $a$-neighborhoods of $S_1$ and $S_2$.

     \item Let $a'\le a$. Let $y\in H_2$ with $\rho(y,s_2)>2a'+1$ for all $s_2\in S_2$. Let $x_1,\ldots,x_{k-1}\in H_1$. Then there is an $x\in H_1$ with $x,y$ having the same $a'$-neighborhoods and such that $\rho(x,x_i)>2d'+1$ for all $1\le i\le k-1$ and $\rho(x,s_1)>2a'+1$ for all $s_1\in S_1$.

     \item Let $a'\le a$. Let $x\in H_1$ with $\rho(x,s_1)>2a'+1$ for all $s_1\in S_1$. Let $y_1,\ldots,y_{k-1}\in H_2$. Then there is an $y\in H_2$ with $x,y$ having the same $a'$-neighborhoods and such that $\rho(y,y_i)>2a'+1$ for all $1\le i\le k-1$ and $\rho(y,s_2)>2a'+1$ for all $s_2\in S_2$.

  \end{enumerate}

Then Duplicator has a winning strategy for $\ehf(H_1,H_2;k)$.

\end{Prop}

The \emph{Distance $k$-round Ehrenfeucht Game} $\dehf(H_1,H_2;k)$ on hypergraphs $H_2$ and $H_2$ is the game $\ehf(H_1,H_2;k)$ with the additional requirement
that for Duplicator to win, she must assure that the distances between corresponding marked vertices are preserved.

We say the $a$-neighborhoods of $x$ and $y$ are \emph{$k$-similar} if Duplicator has a winning strategy for the Distance Ehrenfeucht Game in those neighborhoods that begins with $x$ and $y$ marked 
and has $k$ additional rounds.

Let $x_1,\ldots,x_u\in H_1$ and $y_1,\ldots,y_u\in H_2$ and call these the marked vertices. The $a$-picture of $x_1,\ldots,x_u$ splits into connected components $C_1,\ldots, C_r$
as does the $a$-picture of $y_1,\ldots,y_u$ into $D_1,\ldots,D_{r'}$. Suppose that $r=r'$ and that under a suitable renumbering, $C_i$ and $D_i$ contain corresponding marked vertices.

\begin{Def}

We say that the above $a$-pictures are \emph{$s$-similar} if, in addition to the above conditions, for all pairs of components $C_i$ and $D_i$, duplicator has a winning strategy
for the Distance Ehrenfeucht Game that begins with the $x_l\in C_i$ and $y_l\in D_i$ marked and has $s$ additional rounds.

\end{Def}

Now we can give a powerful extension of the above result.

\begin{Prop} \label{wining strategy}

Set $a=\frac{3^k-1}{2}$. Suppose $H_1$ and $H_2$ have vertex subsets $S_1\subseteq H_1$ and $S_2\subseteq H_2$ with the following properties:

  \begin{enumerate}

     \item $S_1$ and $S_2$ have $k$-similar $a$-neighborhoods.

     \item Let $a'\le a$. Let $y\in H_2$ with $\rho(y,s_2)>2a'+1$ for all $s_2\in S_2$. Let $x_1,\ldots,x_{k-1}\in H_1$. Then there is an $x\in H_1$ with $x,y$ having $k$-similar $a'$-neighborhoods and such that $\rho(x,x_i)>2a'+1$ for all $1\le i\le k-1$ and $\rho(x,s_1)>2a'+1$ for all $s_1\in S_1$.

     \item Let $a'\le a$. Let $x\in H_1$ with $\rho(x,s_1)>2a'+1$ for all $s_1\in S_1$. Let $y_1,\ldots,y_{k-1}\in H_2$. Then there is an $y\in H_2$ with $x,y$ having $k$-similar $a'$-neighborhoods and such that $\rho(y,y_i)>2a'+1$ for all $1\le i\le k-1$ and $\rho(y,s_2)>2a'+1$ for all $s_2\in S_2$.

  \end{enumerate}

Then Duplicator has a winning strategy for $\ehf(H_1,H_2;k)$.

\end{Prop}

   \subsection{Convergence Laws}

Sometimes a zero-one law is too much to ask, so we consider a related weaker concept.

\begin{Def}

A function $p:\mathbb{N}\to[0,1]$ is a \emph{convergence law} if for all $P\in\mathcal{FO}$ the sequence

   $$\mathbb{P}(P)(n,p)$$ converges to a real number in $[0,1]$.

\end{Def}

The convergence laws we will deal with occur when the almost sure theory $T_p$ is not far from
being complete, in the following sense.

Let $T$ be a first order theory of $(d+1)$-uniform hypergraphs and suppose we have a  collection $\mathcal{C}=\{\sigma_1,\sigma_2,\ldots\}$ of first order properties (which is,
of course, at most countable). We say $\mathcal{C}$ is a \emph{complete set
of completions} (relative to $T$) if the following conditions hold:
   
   \begin{enumerate}

           \item $T\cup\{\sigma_i\}$ is complete for all $i$.
           \item For all $i\neq j$, $T\models\lnot(\sigma_i\land\sigma_j)$
           \item For all $i$, the limit $p_i:=\lim_{n\to\infty}\mathbb{P}(\sigma_i)$ exists.
           \item $\sum_{i=1}^\infty p_i=1$.

 \end{enumerate}
 
 In that case, if $H$ is a hypergraph, and $H\models T$, then exactly one of the following
 possibilities hold:

 \begin{gather*}
       H\models T\cup\{\lnot\sigma_1,\lnot\sigma_2,\ldots\}; \\
        H\models T\cup\{\sigma_1\}; \quad
        H\models T\cup\{\sigma_2\};\quad
        H\models T\cup\{\sigma_3\};\quad
        \cdots
          \end{gather*}
Then item $4$ above means that the property axiomatizable by the theory $T\cup\{\lnot\sigma_1,\lnot\sigma_2,\ldots\}$, although not necessarily contradictory, holds 
almost never.

In case $\mathcal{C}=\{\sigma_1,\sigma_2,\ldots\}$ is a complete set of completions for $T$ and $A$ is an elementary property, let $S(A)$ denote the set of indexes $i$ such that $T\cup\{\sigma_i\}\models A$.

\begin{Prop}     \label{complete completions}

Under the above conditions, $\lim_{n\to\infty}\mathbb{P}(A)$ exists for all first order properties $A$ and is given by

            $$\lim_{n\to\infty}\mathbb{P}(A)=\sum_{i\in S(A)}p_i.$$

\end{Prop}

We again refer the reader to Spencer's book \cite{spencer} for a proof.

The next proposition summarizes the above discussion in a way suitable for proving all the convergence laws we need.

\begin{Prop}

If $T_p$ admits a complete set of completions then $p$ is a convergence law.

\end{Prop}

                                              \section{Big-Bang}

                         \subsection{Counting of Butterfly Components}

Now we proceed to investigate zero-one and convergence laws in the early stages of the evolution of $G^{d+1}(n,p)$. More precisely, we investigate edge functions before the double jump:
            $$0\leq p(n)\ll n^{-d}.$$
For functions $p$ in that range, $G^{d+1}(n,p)$ almost surely has no cycles.

\begin{Prop}

If $0\leq p\ll n^{-d}$ then a.a.s. $G^{d+1}(n,p)$ is acyclic. More precisely: if $C$ is a fixed finite cycle, then a.a.s. $G^{d+1}(n,p)$ does not have a copy of 
$C$ as a sub-hypergraph.

\end{Prop}

\begin{proof}

Fix a cycle $C$ with $v$ vertices and $l$ edges. Then $v\leq ld$. Consider the expected number $\mathbb{E}$ of copies of $H$ in $G^{d+1}(n,p)$. Then 
       
           $$\mathbb{E}\sim O(n^vp^l)=o(1)$$
by the upper bound on $p$. By the first moment method, a.a.s. there are no copies of $C$.

\end{proof}

In view of the above, as far as all our present discussions are concerned, the hypergraphs we deal with are disjoint unions of butterflies. Getting more precise information on the statistics of the number of connected components that are finite butterflies of a given order is the most important piece of information to getting zero-one and convergence laws for $p\ll n^{-d}$.

To this end, we define the following random variables. Below $\delta\in\Delta$, where $\Delta$ is the set of all isomorphism classes of butterflies of order $l$ on $v=1+ld$ labelled vertices.

\begin{Def}

 $A^\delta(l)$ is the number of finite butterflies of order $l$ and isomorphism class $\delta$ in $G^{d+1}(n,p)$.

\end{Def}

A butterfly of order $l$ is, in particular, a hypergraph on $v=1+ld$ vertices. Let $c^\delta(l)$ be the number of butterflies of order $l$ and isomorphism class $\delta$ on $v=1+ld$ labelled vertices. To each set $S$ of $v$ vertices on $G^{d+1}(n,p)$ there corresponds 
 the collection of indicator random variables $X_S^1,X_S^2,\ldots,X_S^{c^\delta(l)}$, each indicating that one of the potential $c^\delta(l)$ butterflies of order $l$ and isomorphism class $\delta$ in $S$ is present and is a component. Therefore one has

             $$A(l)=\sum_{S,i}X_S^i,$$
where $S$ ranges over all $v$-sets and $i$ ranges over $\{1,2,\ldots,c(l)\}$.

Note that a connected component isomorphic to a butterfly of a certain isomorphism class is, in particular, an induced copy of that butterfly.
Next we show that the threshold for containment of a butterfly of given order as a connected component is the same for containing butterflies of that order as sub-hypergraphs, not necessarily induced.

In the next proposition, the reader may find the condition 

            $$p\leq C(\log n)n^{-d}$$
  in $2$ rather strange, since it mentions functions outside the scope $p\ll n^{-d}$ of the present chapter. The option to putting this more general proposition here 
reflects the convenience that it has exactly the same proof and that the full condition will be used in the next chapter.

\begin{Prop}

Set $v=1+ld$. The function $n^{-\frac{v}{l}}$ is a threshold for containment of butterflies of order $l$ as components. More precisely:

      \begin{enumerate}

   \item If $p\ll n^{-\frac{v}{l}}$ then a.a.s. $G^{d+1}(n,p)$ has no butterflies of 
order $l$ as sub-hypergraphs. 

   \item If $n^{-\frac{v}{l}}\ll p\leq C(\log n)n^{-d}$ where $C<\frac{d!}{1+ld}$ then, for any $k\in\mathbb{N}$ and $\delta\in\Delta$, a.a.s. $G^{d+1}(n,p)$ has at least $k$ connected components isomorphic the butterfly of order $l$ and isomorphism class $\delta$.

       \end{enumerate}

\end{Prop}

\begin{proof}

Let $\mathbb{E}[A^\delta(l)]$ be the expected value of $A^\delta(l)$. One has ${n\choose v}\sim\frac{n^v}{v!}$ choices of a set of $v$ vertices, $c^\delta(l)$ choices of the butterfly on it, probability $p^l$ that the $l$ vertices of the 
butterfly exist and probability $\sim(1-p)^{v{n\choose d}}\sim\exp(-pv\frac{n^d}{d!})$ that no other edge connects the butterfly to other components.
Therefore

       $$\mathbb{E}[A^\delta(l)]\sim c^{\delta}(l)\frac{n^v}{v!}p^l\exp(-pv\frac{n^d}{d!}).$$

If $p\ll n^{\frac{-v}{l}}$ then $\mathbb{E}[A^\delta(l)]=o(1)$ and, by the first moment method, we have $1.$

For $2$, suppose $n^{-\frac{v}{l}}\ll p\leq C(\log n)n^{-d}$ where $C<\frac{d!}{1+ld}$, so that
 
 $$\mathbb{E}[A^\delta(l)]\to\infty.$$ Let $\mathbb{V}[A^\delta(l)]$ be the variance of $A^\delta(l).$ 
It suffices to show that 
                  $$\mathbb{V}[A^\delta(l)]=o\left(\mathbb{E}[A^\delta(l)]^2\right).$$
Indeed, by the second moment method, the above condition implies that almost surely $A^\delta(l)$ is close to its expectation $\mathbb{E}[A^\delta(l)]\to\infty$.

As $\mathbb{V}[A^\delta(l)]=\mathbb{E}[A^\delta(l)^2]-\mathbb{E}[A^\delta(l)]^2$ and $\mathbb{E}[A^\delta(l)]\to\infty$ the above condition is equivalent to $$\mathbb{E}[A^\delta(l)^2]\sim\mathbb{E}[A^\delta(l)]^2.$$
We have

\begin{align*}
\mathbb{E}[A^\delta(l)^2]  &=
\mathbb{E}\left[\left(\sum X_S^i\right)^2\right] \\
&= \mathbb{E}\left[\sum_{S\cap T=\emptyset,i,j}X_S^iX_T^j\right]+
\mathbb{E}\left[\sum_{S\cap T\neq\emptyset,i,j}X_S^iX_T^j\right].
\end{align*}

As each $X_S^i$ indicates the presence of a butterfly as an isolated component, the second term is zero.  But the first term is 

$$\sim (c^\delta(l))^2\frac{n^{2v}}{v!^2}p^{2l}\exp(-2pv\frac{n^d}{d!})\sim\mathbb{E}\left[A^\delta(l)\right]^2,$$ so we are done.

\end{proof}

Further insight on Proposition 27 gives the following.

\begin{Thrm}

If  $0\le p\ll n^{-(d+1)}$ or there is $l\in\mathbb{N}$ such that      $n^{-\frac{1+ld}{l}}\ll p\ll n^{-\frac{1+(l+1)d}{l+1}}$ then p is a zero-one law.

\end{Thrm}

\begin{proof}

If $0\le p\ll n^{-(d+1)}$, then almost surely there are no edges. As the absence of edges is an elementary property, a model of the almost sure theory in that case is 
the empty hypergraph in a countable number of vertices. So $T_p$ is $\aleph_0$-categorical and all edge functions in that range are zero-one laws.

For $n^{-(1+d)}\ll p\ll n^{-\frac{1+2d}{2}}$, the countable models of the almost sure theory have infinite isolated vertices and infinite isolated edges. This makes $T_p$ $\aleph_0$-categorical
so all such $p$'s are zero-one laws. 

More generally, for $n^{-\frac{1+ld}{l}}\ll p\ll n^{-\frac{1+(l+1)d}{l+1}}$, the countable models of $T_p$ consist of countably many copies of all butterflies of all orders $\le l$ and all isomorphism classes and nothing else.
That makes $T_p$ $\aleph_0$-categorical so these $p$'s are zero-one laws. 

\end{proof}

Let $T_l$ be the first order theory consisting of a scheme of axioms excluding the existence of
cycles and butterflies of order $\ge l+1$ and a scheme that assures the existence of infinite copies
of each type of butterflies of order $\le l$. Then $T_l$ is an axiomatization for $T_p$, where $p$ is
as above.

         \subsection{Just Before the Double Jump}

Consider now an edge function $p$ such that  for all $\epsilon >0$, $n^{-(d+\epsilon)}\ll p\ll n^{-d}$. Such functions would include, for instance, $p(n)=(\log n)^{-1}n^{-d}$.
The countable models of the almost sure theories for such $p$'s must be acyclic and have infinite components isomorphic to butterflies of all orders. But in this range a
new phenomenon occurs: the existence of components that are butterflies of infinite order is left open. There may or there may not be such components, and therefore
the countable models of $T_p$ are not $\aleph_0$-categorical.

We proceed to show that these infinite components do not matter from a first-order perspective, as they will be  ``simulated" by sufficiently large finite components.
Because first-order properties are represented by finite formulae, with finitely many quantifications, this will establish that all models of $T_p$ are elementarily equivalent
in spite of not being $\aleph_0$-categorical.

     \subsubsection{Rooted Butterflies}  \label{Rooted Butterflies}

The results we state in this section for rooted butterflies are stated and proved in Spencer's book \emph{The Strange Logic of Random Graphs} for rooted trees, which are
particular cases of rooted butterflies when $d=1$. The situation is similar to that of the last section: the same arguments in that book apply to the other values of $d$.

A \emph{rooted butterfly} is simply a butterfly $T$ (finite or infinite) with one distinguished vertex $R\in T$, called the \emph{root}. With rooted butterflies, the concepts of \emph{parent, child, ancestor} and \emph{descendent}
are clear: their meaning is similar to their natural computer science couterparts for rooted trees. The \emph{depth} of a vertex is its distance from the root. For each $w\in T$, $T^w$ denotes the sub-butterfly consisting
of $w$ and all its descendants. 

For $r,s\in\mathbb{N}$, we define the $(r,s)$-value of $T$ by induction on $r$. Roughly speaking, we examine the $r$-neighborhood of $R$ and consider any cout greater than $s$, including infinite, indistinguishable 
from each other and call them ``many". Indeed, the possible $(1,s)$-values for a rooted tree $T$ are $0,1,2,\ldots s, M$ where $M$ stands for ``many". The $(1,s)$-value of $T$ is then the number of 
edges incident on the root $R$ if this number is $\le s$. Otherwise, the $(1,s)$-value of $T$ is $M$.

Now suppose the concept of $(r,s)$-value has been defined for all rooted butterflies and denote by $\val(r,s)$ the set of all possible such values. Consider an edge $E=\{R,w_1,\ldots,w_d\}$ of $T$ incident on the root $R$.
The \emph{pattern} of $E$ is the function $P:\val(r,s)\to\{1,2,\ldots,d\}$ such that, for all values $\Omega\in\val(r,s)$, there are exactly $P(\Omega)$ elements in the set $\{T^{w_1},\dots,T^{w_d}\}$ with $(r,s)$-value
$\Omega$. Note that 

                 $$\sum_{\Omega\in\val(r,s)}P(\Omega)=d.$$

Let $\pat(r,s)$ be the set 

$$\left\{P:\val(r,s)\to\{1,\ldots,d\}\mid\sum_{\Omega\in\val(r,s)}P(\Omega)=d\right\}.$$
In other words, $\pat(r,s)$ is the set of all patterns. 

The $(r+1,s)$-value of $T$ is the function $$V:\text{PAT}(r,s)\to\{1,2,\ldots,s,M\}$$ such that, for all
$\Gamma\in\pat(r,s)$, the root $R$ has exactly $V(\Gamma)$ edges incident on it with pattern $\Gamma$, with $M$ standing for ``many". 

Note that for any value $\Omega\in\val(r,s)$ one can easily create a finite rooted butterfly with value $\Omega$: We simply interpret ``many" as $s+1$. Also, any rooted butterfly can be considered a uniform hypergraph
by removing the special designation of the root.

\begin{Prop}

Let $T_1$ and $T_2$ be rooted butterflies with roots $R_1$ and $R_2$ respectively which have the same $(r,s-1)$-value. Then, considering $T_1$ and $T_2$ as graphs, $R_1$ and $R_2$ have
$(sd)$-similar $r$-neighborhoods.

\end{Prop}

Proposition √∑\ref{wining strategy}  gives the following, which can be interpreted as the formal counterpart of the claim that sufficiently large finite butterflies simulate the behavior of infinite ones as far as 
elementary properties are concerned.

\begin{Prop}

Let $H_1$ and $H_2$ be two acyclic graphs in which every finite butterfly occurs as a component an infinite number of times. Then $H_1$ and $H_2$ are elementarily equivalent.

\end{Prop}

It is convenient to emphasize that, above, $H_1$ and $H_2$ may or may not have infinite components.

\begin{Thrm}

Suppose $p$ is an edge function satisfying, for all $\epsilon>0$, 
                           $$n^{-(d+\epsilon)}\ll p\ll n^{-d}.$$
    Then $p$ is a zero-one law.

\end{Thrm}

\begin{proof}

Consider an edge function $p$ such that $n^{-(d+\epsilon)}\ll p\ll n^{-d}$ for all $\epsilon>0$. We see that all countable models of $T_p$ satisfy the hypotheses of the above proposition. 
Therefore they are elementarily equivalent and these $p$'s are zero-one laws.

\end{proof}

Let $T$ be the first order theory consisting of a scheme of axioms excluding the existence of cycles and a scheme that assures that every finite butterfly of any order appears as a component
an infinite number of times. Then $T$ is an axiomatization for $T_p$. 

          \subsection{On the Thresholds}

So far, we have seen that if $p$ satisfies one of the following conditions

\begin{enumerate}

   \item $0\le p\le n^{-(d+1)}$
   \item $n^{-\frac{1+ld}{l}}\ll p\ll n^{-\frac{1+(l+1)d}{l+1}}$, for some $l\in\mathbb{N}$
   \item $n^{-(d+\epsilon)}\ll p\ll n^{-d}$ for all $\epsilon>0$

\end{enumerate}
then $p$ is a zero-one law.

An $L$-function $p$ in the range $0\le p\ll n^{-d}$ that violates all the above three conditions must satisfy, for some $l\in\mathbb{N}$ and $c\in(0,+\infty)$, the condition

                    $$p(n)\sim c\cdot n^{-\frac{1+ld}{l}}.$$

Informally speaking, in that range, an $L$-function that is not ``between" the thresholds is ``on" some threshold. In that case, $p$ is not a zero-one law. Our next goal is to show that those $p$'s are still convergence laws

     \subsubsection{Limiting Probabilities on the Thresholds} 

Let $l\in\mathbb{N}$ and let $T_1,T_2,\ldots,T_u$ denote the collection of all possible (up to isomorphism) butterflies of order $l$ and let $I$ be the set of all $u$-tuples $\textbf{m}=(m_1,\ldots,m_u)$ of non-negative integers.
Finally, let $\sigma_{\textbf{m}}$ be the elementary property that there are precisely $m_i$ components $T_i$ for $i\in\{1,\ldots,u\}$.

\begin{Prop}

Let $p\sim c\cdot n^{-\frac{1+ld}{l}}$. Then the collection $\{\sigma_{\textbf{m}}|\textbf{m}\in I\}$ is a complete set of completions for $T_p$. In particular, $p$ is a convergence law.

\end{Prop}

\begin{proof}

We show properties $1,2,3$ and $4$ in the definition of a complete set of completions.

The countable models of $T_p\cup\{\sigma_{\textbf{m}}\}$ have no cycles, a countably infinite number of components of each butterfly of order $\le l-1$, no sub-hypergraph isomorphic to a butterfly of order $\ge l+1$ and exactly
$m_i$ components $T_i$ for each $i$. So $T_p\cup\{\sigma_{\textbf{m}}\}$ is $\aleph_0$-categorical and, in particular, complete, so we have property $1$. 

Tautologically no two of the $\sigma_{\textbf{m}}$ can hold simultaneously, so we have property $2$.

For each $i\in\{1,2,\ldots,u\}$, let $\delta_i$ be the isomorphism type of $T_i$. For notational convenience, set $c_i:=c^{\delta_i}(l)$ and $A_i:=A^{\delta_i}(l)$. 
The next lemma implies properties $3$ and $4$ and, therefore, completes the proof.

\end{proof}

\begin{Lemma}  \label{poisson1}

In the conditions of the above proposition, the random variables $A_1,A_2,\ldots,A_u$ are asymptotically independent Poisson with means $\lambda_1=\frac{c_1}{v!}c^l,\lambda_2=\frac{c_2}{v!}c^l,\ldots,\lambda_u=\frac{c_u}{v!}c^l$. That is to say,

       $$p_{\textbf{m}}:=\lim_{n\to\infty}\mathbb{P}(\sigma_{\textbf{m}})=\prod^u_{i=1}e^{-\lambda_i}\frac{\lambda_i^{m_i}}{m_i!}.$$
In particular  $$\sum_{\textbf{m}\in I}p_{\textbf{m}}=1.$$

\end{Lemma}

\begin{proof}

By the method of factorial moments, is suffices to show that, for all $r_1,r_2,\ldots,r_u\in\mathbb{N}$ we have

                    $$\mathbb{E}\left[(A_1)_{r_1}\cdots(A_u)_{r_u}\right]\to\lambda^{r_1}\cdots\lambda^{r_u}.$$

As we have seen, each $A_i$ can be written as a sum of indicator random variables $A_i=\sum_{S,j} X^{i,j}_{S}$, each $X^{i,j}_S$ indicates the event $E_S^{i,j}$ that the $j$-th of the potential
butterflies on the vertex set $S$ is present and is a component. Then
$$\mathbb{E}\left[(A_1)_{r_1}\cdots(A_u)_{r_u}\right]=
\sum_{S_1,\ldots,S_u,j_1,\ldots,j_u}\mathbb{P}[E_{S_1}^{1,j_1}\land\ldots
\land E^{u,j_u}_{S_u}].$$
The above sum splits into $\sum_1+\sum_2$ where $\sum_1$ consists of the terms with $S_1,\ldots,S_u$ pairwise disjoint. As each $X^{i,j}_S$ indicates the presence of a butterfly as a component,
we have $\sum_2=0$. On the other hand, it is easy to see that if $p\sim c\cdot n^{-\frac{1+ld}{l}}$ then

                   $\sum_1\sim\prod_i c_i^{r_i}\frac{n^{r_iv}}{v!}p^{r_il}\exp(-pr_ivn^d/d!)\sim\prod_i\lambda^{r_i}$, so we are done.

\end{proof}

It is worth noting that if $a_i$ is the number of automorphisms of the butterfly whose isomorphism type is $\delta_i$ then one has $\frac{c_i}{v!}=\frac{1}{a_i}$.

The convergence laws we got so far provide a nice description of the component structure in the early history of $G^{d+1}(n,p)$: it begins empty, then isolated edges appear, then all butterflies of order two, then all of order three, and so on untill $\ll n^{-d}$, immediately before the double jump takes place.

In what follows, $\bb$ stands for ``Big-Bang".

\begin{Def}

 $\bb$ is the set of all $L$-functions $p:\mathbb{N}\to[0,1]$ satisfying $0\le p\ll n^{-d} $.

\end{Def}

Now it is just a matter of putting pieces together to get the following.

\begin{Thrm}

All elements of $\bb$ are convergence laws.

\end{Thrm}

\begin{proof}

Just note that any $L$-function on the above range satisfies one of the following conditions:

\begin{enumerate}

   \item $0\le p\le n^{-(d+1)}$
   \item $n^{-\frac{1+ld}{l}}\ll p\ll n^{-\frac{1+(l+1)d}{l+1}}$, for some $l\in\mathbb{N}$
   \item $n^{-(d+\epsilon)}\ll p\ll n^{-d}$ for all $\epsilon>0$
   \item $p\sim c\cdot n^{-\frac{1+ld}{l}}$ for some constant $c\in(0,+\infty)$

\end{enumerate}

\end{proof}


It is worth noting that the arguments used in getting zero-one laws for the intervals
\begin{enumerate}

   \item $0\le p\le n^{-(d+1)}$
   \item $n^{-\frac{1+ld}{l}}\ll p\ll n^{-\frac{1+(l+1)d}{l+1}}$, for some $l\in\mathbb{N}$
   \item $n^{-(d+\epsilon)}\ll p\ll n^{-d}$ for all $\epsilon>0$
   
 \end{enumerate}
 do not require the edge functions to be in Hardy's class, so \emph{all} functions inside those intervals are zero-one laws, regardless of being $L$-functions.
 
 On the other hand, taking $p=c(n)\cdot n^{-\frac{1+ld}{l}}$, where $c(n)$ oscillates infinitely
 often between two
 different  positive values is sufficient to rule out a convergence law for that edge function.
    
 Also, our discussion implies that, in a certain sense, most of the functions in $\bb$ are zero-one laws: the only way one of that functions can avoid this condition is being inside
one of the countable windows inside a threshold of appearance of butterflies of some order. 

In the following sections, similar pieces of reasoning will yield an analogous result for another interval of edge functions.

                               \section{The Double Jump}

            In this section, we consider the random hypergraph $G^{d+1}(n,p)$, where $p\sim\frac{\lambda}{n^d}$ for some constant $\lambda>0$. Of course, this is equivalent to having 
 $p=\frac{\lambda_n}{n^d}$, where $\lambda_n\to\lambda$.
Our goal is to show that the above $p$'s are convergence laws.            
             
             Simple applications of  Theorem \ref{vantsyan} 
imply that, in this range, the countable models of the almost sure theory have infinitely many 
connected components isomorphic to each finite butterfly and no bicyclic (or more) components.
Also, there possibly are infinite butterflies as components. As we have already seen, these
components do not matter from a first order perspective, as they are simulated by sufficiently
large finite butterflies. More precisely: the addition of components that are infinite butterflies do not
alter the elementary type of a hypergraph that has infinitely many copies of each finite butterfly.

         The above considerations suggest that it may be useful to consider the asymptotic
 distribution of the various types of unicyclic components in $G^{d+1}(n,p)$. It follows that the 
 distributions are asymptotic independent Poisson, so that we have a complete set of completions
 of the almost sure theory. The procedure is a bit more involved then the one in the above section
 because we will actually need one completion for each fixed maximal quantification depht 
 $s\in\mathbb{N}$.

        \subsection{Values and Patterns}

        Recall from section \ref{Rooted Butterflies} the definitions of the $(r,s)$-value of a rooted
  butterfly and of the $(r,s)$-pattern of a rooted edge. 
  
  \begin{Def}
  
  Fix a vertex $v$ in the random hypergraph $G^{d+1}(n,p)$ and a value $\Omega\in\val(r,s)$. Then $p_{\Omega}^n$ is the 
  probability that the ball of center $v$ and radius $r$ is a butterfly of value $\Omega$, considering
  $v$ as the root.
  
  Similarly, for an edge $E=\{v,v_1,\ldots,v_d\}\in G^{d+1}(n,p)$ and a pattern $\Gamma\in\pat(r,s)$, $p_{\Gamma}^n$ is the probability that the pattern of $E$ is $\Gamma$, considering $v$ as the
  root.

  \end{Def}

         Next, we proceed to describe the asymptotic behavior of $p_{\Omega}^n$ and 
   $p_{\Gamma}^n$ as $n\to\infty$.

          \subsection{Poisson Butterflies}

       Now we consider a random procedure for constructing a rooted butterfly. In fact, it is a simple 
modification of the Galton-Watson Branching Process, aiming to fit the case of hypergraphs.

        $\text{B}(r,\mu)$ is the random rooted butterfly constructed as follows:
        
        Let $P(\mu)$ be the Poisson distribution with mean $\mu$ and start with the root $v$. The number of edges incident on $v$ is $P(\mu)$. Each child of $v$ has, in turn, further $P(\mu)$ edges
incident on it (there being no further adjacencies among them, so as to $\text{B}(r,\mu)$ remain
a butterfly). Repeat the process until the end of the $r$-th and then halt. The resulting structure
is a random butterfly rooted on $v$.
           
           One obtains a similar structure $\tilde{\text{B}}(r,\mu)$ beginning with an edge $E$, rooted 
on $v$, and requiring that each non-root vertice has $P(\mu)$ further edges and so on, until the
$r$-th generation.

            For $\Omega\in\val(r+1,s)$ be a $(r+1,s)$-value, let $p_{\Omega}$ be the probability that
$B(r+1,s)$ has value $\Omega$. Similarly, if $\Gamma\in\pat(r,s)$ is a $(r,s)$-pattern, let 
$p_{\Gamma}$ be the probability that $\tilde{\text{B}}(r,s)$ has pattern $\Gamma$. Note that the
method of factorial moments implies that if 
                        $$\pat(r,s)=\{\Gamma_1,\Gamma_2\ldots,\Gamma_N\}$$
then the distributions of edges incident on the root $v\in B(r+1,\mu)$ with pattern $\Gamma_i\in\pat(r,s)$  are poisson $P(p_{\Gamma_i}\cdot\mu)$, \emph{independently} for each $i\in\{1,\ldots,N\}$.

            The most important piece of information to showing that $p\sim\lambda n^{-d}$ are 
 are convergence laws is that
            $$p_{\Omega}^n\to p_{\Omega}$$ and
            $$p_{\Gamma}^n\to p_{\Gamma}$$
 for every value $\Omega$ and every pattern $\Gamma$, with $\mu=\frac{\lambda}{d!}$.

           \subsection{Size of Neighborhoods}

        The next lemma shows that the probability that the size of the neighborhood of a given vertice
is large is $o(1)$. There, $|B(v_0,r)|$ is the number of vertices in the ball of center $v_0$ and radius $r$.

\begin{Lemma}   \label{size}

Fix $\epsilon>0$, $\delta>0$ and $r\in\mathbb{N}$. Then

  $$\mathbb{P}(\exists v_0 , | B(v_0,r)|>\epsilon n^{\delta})\to 0.$$

\end{Lemma}

\begin{proof}

We proceed by induction on $r$.

   If $r=0$, there is nothing to prove.

   For $r=1$, first fix $\tilde{\lambda}>\lambda$. Then one has, for sufficiently large $n$,

   \begin{align*}
 \mathbb{P}(\exists v_0 , |B(v_0,1)|>\epsilon n^{\delta})&\leq n\cdot\mathbb{P}(|\{\text{neighbohrs of } v_0\}|>\epsilon n^{\delta})  \\
 &=n\cdot\mathbb{P}(|\{\text{edges on }v_0\}|>\frac{\epsilon n^{\delta}}{d}) \\
 &=n\cdot\sum_{l>\frac{\epsilon n^{\delta}}{d}}\mathbb{P}(|\{\text{edges on } v_0\}|=l) \\
 &\leq n\cdot\sum_{l>\frac{\epsilon n^{\delta}}{d}}\frac{1}{d!}{n\choose{d}}^{l}p^l(1-p)^{n\choose{d}} \\   
 &\leq n\cdot\sum_{l>\frac{\epsilon n^{\delta}}{d}}n^{dl}\cdot\frac{\lambda_n^l}{n^{dl}}\exp[-n^d\cdot\frac{\lambda_n}{n^d}] \\
 &\leq(\text{constant})\frac{\tilde{\lambda}^{\frac{\epsilon n^{\delta}}{d}}}{(\epsilon n^{\delta}d^{-1})!}\cdot\frac{n}{1-\epsilon^{-1}\tilde{\lambda}dn^{-\delta}} \\
 &\sim(\text{constant})\frac{\tilde{\lambda}^{\epsilon n^d}\cdot n}{\sqrt{2\pi\epsilon n^{\delta}d^{-1}}\cdot(\epsilon n^{\delta}d^{-1}e^{-1})^{\frac{\epsilon n^{\delta}}{d} }} \\
 &=o(1) .
   \end{align*}

   For the induction step, note first that the induction hypothesis implies that, almost surely, every ball of
radius $r$ has size at most $\sqrt{\epsilon}n^{\delta/2}$. Let $B$ be the event 
                    $$\{\exists v_0 , |B(v_0,r+1)|>\epsilon n^{\delta}\}.$$
Note that
      $$B\implies\{\exists v_0 , |B(v_0,r)|>\sqrt{\epsilon}n^{\delta/2}\}.$$
Therefore
      $$\mathbb{P}(B)\leq\mathbb{P}(\exists v_0 , |B(v_0,r)|>\sqrt{\epsilon}n^{\delta/2})=o(1)$$
so we are done.

\end{proof}

                                                  \subsection{$s$-Completions}

    In the present range $p\sim\frac{\lambda}{n^d}$, the almost sure theory has one complete
    set of completions for each fixed quantifier depth $s$.

      \begin{Def}

    Let $T$ be a first-order theory.                   
    A family $\{\sigma_f\}_{f\in I}$ of elementary sentences is a \emph{complete set of $s$-completions for $T$} if the following conditions hold:
    
    \begin{enumerate}
    
         \item For all $f\in I$ and any elementary sentence $A$ of quantifier depth at most $s$,
   one has either $T\cup\{\sigma_f\}\models A$ or $T\cup\{\sigma_f\}\models\lnot A$.  
   
          \item If $f,g\in I$ are such that $f\neq g$ then $T\models\lnot(\sigma_f\land\sigma_g)$.
          
          \item For all $f\in I$, $p_{\sigma_f}:=\lim_{n\to\infty}\mathbb{P}(\sigma_f)$ exists.
          
          \item $\sum_{f\in I}p_{\sigma_f}=1$.
    
    \end{enumerate}

    \end{Def}

    For a complete set of $s$-completions $\{\sigma_f\}_{f\in I}$ of $T$ and an elementary sentence
    $A$ of quantifier depth at most $s$, let $S(A)$ be the set of indices $f\in I$ such that
    $T\cup\{\sigma_f\}\models A$.

    \begin{Prop}
    
    If $\{\sigma_f\}_{f\in I}$ is a complete set of $s$-completions for the almost sure theory $T$ of an
    edge function $p$, then for all elementary sentence $A$ of quantifier depth at most $s$, the 
    limit $\lim_{n \to\infty}\mathbb{P}(A)$ exists and one has
              $$\lim_{n \to\infty}\mathbb{P}(A)=\sum_{f\in S(A)}p_{\sigma_f}.$$

    \end{Prop}

    The proof is similar to that of Proposition \ref{complete completions}. Obviously, one has the
    following:

    \begin{Cor}   \label{s-completions}
    
    If the almost sure theory of $p$ admits, for every $s\in\mathbb{N}$, a complete set of complete
    set of $s$-completions, then $p$ is a convergence law.
    
     \end{Cor}

             \subsection{Main Results}

In the following, $\Omega$ is a $(r,s)$-type and, for a hypergraph $C$, $E(C)$
is the edge set of $C$.

 \begin{Lemma}

 Fix vertices $v_1,v_2,\ldots,v_{kd}$ and an unicyclic connected configuration $C$ on the vertice
 set $\{v_1,v_2,\ldots,v_{kd}\}$ (necessarily with $k$ edges).
 Then, given $C$, the probability that the ball $B(v_1,r)$ of center $v_1$ and radius $r$ on 
 $G^{d+1}(n,p)\setminus E(C)$ 
 is a butterfly of type $\Omega$ is $p_{\Omega}+o(1)$.

  \end{Lemma}            
             
   \begin{proof}
   
   All probabilities mentioned on this proof are conditional on getting $C$.
   
   By induction on $r$, we show that, for all $0\leq\delta<1$, $0<\epsilon$ and $r\in\mathbb{N}$,
   and given $C$, the ball $B(v_1,r)\subseteq G^{d+1}(n-\epsilon n^{\delta},p)\setminus E(C)$ is a butterfly of 
   type $\Omega$ with probability $p_\Omega+o(1)$, and, similarly, that given an edge $E$ on $v_1$, the pattern of $E$ in $G^{d+1}(n-\epsilon n^{\delta},p)\setminus E(C)$
is $\Gamma$ with probability $p_\Gamma+o(1)$.   

For $r=1$, let $\mathcal{E}$ be the edge set of 
      $$B(v_1,r)\subseteq G^{d+1}(n-\epsilon n^{\delta},p)\setminus E(C).$$
   Then one has
   
   \begin{align*}
   \mathbb{P}(\mathcal{E}=l)&\sim\frac{1}{l!}{n\choose d}^l\cdot p^l\cdot(1-p)^{n\choose d} \\
   &\sim\frac{n^{dl}}{l!(d!)^l}\cdot(\frac{\lambda}{n^d})^l\cdot\exp[-p\frac{n^d}{d!}] \\
   &=\frac{1}{l!}(\frac{\lambda}{d!})^l\cdot\exp[-\frac{\lambda}{d!}]
   \end{align*}
    so that $\mathcal{E}$ has asymptotic distribution $P(\frac{\lambda}{d!})$, which agrees to $B(1,\frac{\lambda}{d!})$.
    A similar argument applies in the case of patterns of edges.
   
   For the induction step, let $\Gamma_1,\ldots,\Gamma_N$ be the possible $(r,s)$ patterns of 
   edges. We use the method of factorial moments to show that the distributions of the various
   $(r,s)$-patterns are asymptotically Poisson independent with means 
   $\frac{\lambda^l}{l!}\cdot p_{\Gamma_i}$, for $1\leq i\leq N$, which clearly suffices.
   First we show that the expected number $\mathcal{P}$ of pairs of edges $E_i,E_j$ on $v_1$
   with patterns $\Gamma_i$ and $\Gamma_j$ respectively is asymptotically to
                     $$\frac{\lambda^{2l}}{(l!)^2}\cdot p_{\Gamma_i}\cdot p_{\Gamma_j}.$$ 
     The general case is similar, with more cumbersome notation.
     
     Let $p_0^n$ be the probability of the event $A$ that the pattern of $E_i$ is $\Gamma_i$ and 
    that of $E_j$ of $\Gamma$. Define $\tilde{p}_0^n$ to be the probability of the event $B$ that
    the pattern of $E_i$ is $\Gamma_i$ and the pattern of $E_j$, \emph{not counting the vertices
    already used in $E_i$}. 
    By Lemma \ref{size} and the induction hypothesis, one has
          $$\tilde{p}_0^n\to p_{\Gamma_i}\cdot p_{\Gamma_j}.$$ 
          Note that $p_0^n\sim \tilde{p_0}^n$, because any hypergraph on the symetric difference $A\triangle B$ 
    has at least two cycles, so that $\mathbb{P}(A\triangle B)=o(1)$. So we have

    \begin{align*}
    \mathbb{E}[\mathcal{P}]&\sim{n\choose d}^2\cdot p^2\cdot p_0^n \\
    &\sim(\frac{n^d}{d!})^2\cdot p^2\cdot\tilde{p}_0^n \\
    &\sim\frac{\lambda^2}{d!}\cdot p_{\Gamma_i}\cdot p_{\Gamma_j}
    \end{align*} 
    
    and we are done.

 \end{proof}

 Arguments similar to the one above show that, defining the $(r,s)$-pattern of a cyclic configuration 
 $C$ in the obvious manner, the distribution of cycles of the various types are asymptotically
 Poisson independent. Therefore, if $\Delta_1,\ldots,\Delta_M$ are the possible $(r,s)$-patterns
 of cycles and $\mathcal{P}_i$ is the number of cycles of pattern $\Delta_i$, then
        $$(\mathcal{P}_1=l_1)\land\ldots\land(\mathcal{P}_M=l_M)$$
    is an elementary property and   
       $$\{(\mathcal{P}_1=l_1)\land\ldots\land(\mathcal{P}_M=l_M) | i_1,\ldots,i_M\in\mathbb{N}\}$$
   is a complete set of $s$-completions for the almost sure theory of $p$. So, in view of
   Corollary \ref{s-completions}, we obtain the following.

   \begin{Thrm}
   
   If $p\sim\frac{\lambda}{n^d}$, then $p$ is a convergence law.
   
    \end{Thrm}

                                                       \section{Big-Crunch}

The present chapter is devoted to getting a result analogous to the ones above on another interval of edge functions, immediately after the double jump. We call that interval $\bc$, for Big-Crunch, because, informally, when ``time" (the edge functions $p$) flows forth, the behavior of the complement of the giant component is the same of the behavior $G^{d+1}(n,p)$ assumes in the Big-Bang $\bb$ with time flowing backwards.

More concretely, $\bc$ is the set of $L$-functions $p$ satisfying $n^{-d}\ll p\ll n^{-d+\epsilon}$ for all $\epsilon>0$. An important function inside this interval is $p=(\log n)n^{-d}$ which will be seen, in the next chapter, to be
the threshold for $G^{d+1}(n,p)$ to be connected. In the subintervals $n^{-d}\ll p\ll (\log n)n^{-d}$ and $(\log n)n^{-d}\ll p\ll n^{-d+\epsilon}$, nothing interesting happens in the first order perspective. This will imply that these intervals are entirely made of zero-one laws.

Inside the window $p\sim C\cdot(\log n)n^{-d}$ (with $C$ some positive constant), very much the opposite is true: here we find an infinite collection of thresholds of elementary properties and also an infinite collection of zero-one and convergence laws.

\subsection{Just Past the Double Jump}

Consider the countable models of the almost sure theory $T_p$ with 

     $$n^{-d}\ll p\ll(\log n)n^{-d}.$$

 As we have already seen, in that range we still have components isomorphic to all
 finite butterflies of all orders and the possibility of infinite butterflies is still open.    
 The threshold for the appearance of small sub-hypergraphs excludes the possibility of bicyclic  
 (or more) components. By the same reason, we have components with cycles of all types.
 The following shows, in particular, that vertices of small degree do not occur near the cycles.
 
 \begin{Prop}\label{unicyclic}

Suppose $p\gg n^{-d}$. Let $H$ be a finite connected configuration with at least one cycle and at least one vertex of small degree. Then the expected number of such configurations in $G^{d+1}(n,p)$ is $o(1)$.
In particular a.a.s. there are no such configurations.

\end{Prop}

\begin{proof}

Let the configuration $H$ have $v$ vertices and $l$ edges. As $H$ is connected and has at least one cycle, we have $v\leq ld$. For convenience, set $\alpha=\frac{pn^d}{d!}$. Note that $\alpha\to +\infty$.
Let $\mathbb{E}$ be the expected number of configurations $H$. Then

$$\mathbb{E}= O\left(\frac{n^v}{v!}p^l(1-p)^{\frac{n^d}{d!}}\right)= O\left(\frac{n^v}{v!}p^l\exp(-p\frac{n^d}{d!})\right)=$$

$$=O\left(n^{dl}p^l\exp(-p\frac{n^d}{d!})\right)t=O\left(\alpha^l\exp(-\alpha)\right)=o(1).$$

The ``in particular" part follows from the first moment method.

\end{proof}

Now it is easy to see that the edge functions in the present range are zero-one laws.

\begin{Thrm}

Suppose $p$ is an edge function satisfying
                $$n^{-d}\ll p\ll(\log n)n^{-d}.$$
        Then $p$ is a zero-one law.
               
\end{Thrm}

\begin{proof}

By proposition \ref{unicyclic}, every vertex in the union of all the unicyclic components has infinite neighbors. This determines these components up to isomorphism and it does not pay for Spoiler to play there.
But in the complement of the above set, we have already seen that Duplicator can win all
$k$-round Ehrenfeucht Games. Therefore all countable models of $T_p$ are elementarily
equivalent and these $p$ are zero-one laws.

\end{proof}

We note that the non-existence of vertices of small 
degree near cycles is first-order axiomatizable. For each $l,s,k\in\mathbb{N}$ there is a first order sentence which excludes all of the (finitely many)
 configurations with cycles of order $\le l$ at distance $\le s$ from one vertex of degree $\le k$. Similar considerations show that the non-existence of bicyclic (or more) components is also first-order
axiomatizable. So one easily gets a simple axiomatization for the almost sure theories of the
above edge functions.

\subsection{Beyond Connectivity}

Now we consider countable models of $T_p$ with 

                     $$(\log n)n^{-d}\ll p\ll n^{-d+\epsilon}$$
  for all positive $\epsilon$.
  
  Again, the thresholds for appearance of small sub-hypergraphs imply that, in this range, we 
  have all cycles of all types as sub-hypergraphs, and no bicyclic (or more) components.
  Now all vertices of small degree are gone.
  
  \begin{Prop}
  
  For $p\gg (\log n)n^{-d}$, the expected number of vertices of small degree in $G^{d+1}(n,p)$
  is $o(1)$. In particular, a.a.s. there are no vertices of small degree.
  
   \end{Prop}

   \begin{proof}
   
   Fix a natural number $k$ and let $\mathbb{E}$ be the expected number of vertices of degree
   $k$ in $G^{d+1}(n,p)$.  Then
   
        $$\mathbb{E}\sim n(1-p)^{\frac{n^d}{d!}}\sim n\exp\left({-p\frac{n^d}{d!}}\right)=o(1).$$
   
   \end{proof} 
   
\begin{Thrm} \label{beyond conectivity}

Let $p$ be an edge function satisfying
              $$(\log n)n^{-d}\ll p\ll n^{-d+\epsilon}$$
         Then $p$ is a zero-one law.     
              
\end{Thrm}
   
\begin{proof}

   The countable models of $T_p$ have components that contain cycles of all types, no bicyclic 
   (or more) components and may possibly have butterfly components. As no vertex can have
   small degree, all vertices in that components have infinite neighbors, so these components
   are unique up to isomorphism. But $T_p$ is not $\aleph_0$-categorical since the existence of butterfly components is left open.
 Proposition √∑\ref{wining strategy}  gives that these models are elementarily equivalent, so these $p$
 are zero-one laws. 
 
 \end{proof}
 
 The discussion found on the proof of theorem \ref{beyond conectivity}  also gives simple axiomatizations for the almost sure theories
 of the above $p$.

         \subsection{Marked Butterflies}  
         
         Now we are left to the case of $L$-functions $p$ comparable to $n^{-d}\log n$. In other words, to complete our discussion, we must get a description of what happens when an edge function $p$ is such that $n^d p\slash \log n$ tends to a finite constant $C\neq 0$.
         
         In the last section, the counting of the connected components isomorphic to butterflies was the 
   fundamental piece of information in the arguments that implied all the convergence laws we found there. Copies of 
   butterflies as connected components are, in particular, induced such copies.
         
       It turns out that the combinatorial structure whose counting is fundamental to getting
  the convergence laws in the window $p\sim C\cdot\frac{\log n}{n^d}$ is still that of butterflies, but
  now the copies are not necessarily induced. Instead, some vertices receive
  markings, meaning that those vertices must have no further neighbors than those indicated on
  the ``model" butterfly. On the non-marked vertices no such requirement is imposed: they are
  free to bear further neighbors. These copies of butterflies are then, in a sense, ``partially induced".

\begin{Def}

Let $v^*,l\in\mathbb{N}$. A $v^*$-marked $l$-\emph{butterfly} is a finite connected (Berge)-acyclic hypergraph with $l$ edges and with $v^*$ distinguished vertices, called the \emph{marked} vertices.

\end{Def}

Note that a $v^*$-marked $l$-butterfly is a hypergraph on $v=1+ld$ vertices.

\begin{Def}

Let $B$ be a $v^*$-marked $l$-butterfly and $H$ be a hypergraph. A \emph{copy} of $B$ in $H$ is a (not necessarily induced) sub-hypergraph of $H$ isomorphic to $B$ where if $w$ is a marked vertex of $B$ and $w'$ is the corresponding vertex of $H$ under the above isomorphism, then $w$ and $w'$ have the same degree.

\end{Def}

An edge of a Berge-acyclic hypergraph incident to exactly one other edge is called a \emph{leaf}.

\begin{Def}

A $v^*$-marked $l$-butterfly is called \emph{minimal} if every leave has at least one marked vertex.

\end{Def}

Now, the most important concept to understanding the zero-one and convergence laws on $\bc$ is the counting of minimal marked butterflies.

\begin{Def}

Let $\Gamma$ be the finite set of all isomorphism types of minimal $v^*$-marked $l$-butterflies on $1+ld$ labelled vertices and fix $\gamma\in\Gamma$.
Then  $c(l,v^*,\gamma)$ is the number of possible $v^*$-marked $l$-butterflies of isomorphism type $\gamma$ on $1+ld$ labelled vertices.

\end{Def}

\begin{Def}

The random variable $A(l,v^*,\gamma)$ is the number of copies of $v^*$-marked $l$-butterflies of isomorphism type $\gamma$ in $G^{d+1}(n,p)$.

\end{Def}

                      \subsubsection{Counting of Marked Butterflies}
                      
 Now we use the first and second moment methods to get precise information on the 
 counting of minimal marked butterflies for edge functions on the range
  
                           $$p\sim C\cdot\frac{\log n}{n^d} \ , \  C>0.$$
                           
Rather informally, when the coefficient of $\frac{\log n}{n^d}$ in $p$ avoids the rational value
$\frac{d!}{v^*}$ then the expected number of $v^*$-marked butterflies is either $0$ or $\infty$. The first moment method implies that, in the first case, a.a.s. there are no $v^*$-marked butterflies.
The second moment method will yield that, in the second case, there are many such minimal marked butterflies.

If $C=\frac{d!}{v^*}$, then knowledge of more subtle behavior of the edge function is required:
we are led to consider the coefficient $\omega$ of $\frac{\log\log n}{n^d}$ in $p$. If this coefficient
avoids the integer value $l$ then the expected number of $v^*$-marked $l$-butterflies is either
$0$ or $\infty$. Again, first and second moment arguments imply that, in the first case, the number 
of such butterflies is a.a.s. zero and, in the second case, the number of such minimal butterflies is very large.

Finally, if $\omega=l$, then knowledge of even more subtle behavior of the edge function is 
required: we consider the coefficient $c$ of $\frac{1}{n^d}$ in $p$. If this coefficient diverges,
then the expected number of $v^*$-marked $l$-butterflies is either $0$ or $\infty$ and, again,
first and second moment methods imply that the actual number of such butterflies is what 
one expects it to be.

All above cases give rise to zero-one laws. The remaining case is the one when the coefficient
$c$ converges. In this case, the fact that the almost sure theories are almost complete
will yield convergence laws.


Let $p=p(n)$ be comparable to $\frac{\log n}{n^d}$. That is, let $\frac{n^d p}{\log n}$ converge to a constant $C\neq 0$.

\begin{Prop}

Fix $\gamma\in\Gamma$.

  \begin{enumerate}

             \item If $C<\frac{d!}{v^*}$ then $\mathbb{E}[A(l,v^*,\gamma)]\to +\infty$.

             \item If $C>\frac{d!}{v^*}$ then $\mathbb{E}[A(l,v^*,\gamma)]\to 0$ for all $l\in\mathbb{N}$.

  \end{enumerate}

In particular, if $C>\frac{d!}{v^*}$ then, for any $l\in\mathbb{N}$, a.a.s. $A(l,v^*,\gamma)=0$.

\end{Prop}

\begin{proof}

Set $c=c(l,v^*,\gamma)$ and $v=1+ld$.

Note that $pv^*\frac{n^d}{d!}\sim Cv^*\frac{\log n}{d!}$ so $pv^*\frac{n^d}{d!}- Cv^*\frac{\log n}{d!}=o(1)\log n$. Therefore one has

\begin{align*}
\mathbb{E}[A(l,v^*,\gamma)]&\sim 
c\frac{n^v}{v!}p^l(1-p)^{v^*\frac{n^d}{d!}} \\
&\sim c\frac{n^v}{v!}p^l\exp\left(-p v^*\frac{n^d}{d!}\right) \\
&\sim c\frac{n^v}{v!}p^l\exp\left(o(1)\log n-Cv^*\frac{\log n}{d!}\right) \\
&\sim c\frac{n^v}{v!}(C\log n)^l n^{-ld}\exp\left(o(1)\log n-Cv^*\frac{\log n}{d!}\right) \\
&\sim\frac{c}{v!}(C\log n)^l n^{1-\frac{Cv^*}{d!} + o(1)},
\end{align*}

and the result follows.

The ``in particular" part follows from the first moment method.

\end{proof}

 Now consider $p\sim\frac{d!}{v^*}\cdot\frac{\log n}{n^d}$ so that $v^*n^d\frac{p}{d!}-\log n=o(1)\log n$ and let 

$$\omega(n)=\frac{v^*n^d\frac{p}{d!}-\log n}{\log\log n}.$$ 

\begin{Prop}

 Fix $l\in\mathbb{N}$ and $\epsilon>0$.
 
                       \begin{enumerate}
           
          \item If eventually $\omega<l-\epsilon$ then $\mathbb{E}[A(l,v^*,\gamma)]\to +\infty$

          \item If eventually $\omega>l+\epsilon$ then $\mathbb{E}[A(l,v^*,\gamma)]\to 0$.

                        \end{enumerate}

In particular, the second condition implies that a.a.s. $A(l,v^*,\gamma)=0$.

\end{Prop}

\begin{proof} 

Set $c:=c(l,v^*,\gamma)$ and $v:=1+ld$.

Note that  $$v^*n^d\frac{p}{d!}=\log n+\omega\log\log n$$ 
so one has

\begin{align*}
\mathbb{E}[A(l,v^*,\gamma)]&\sim \frac{c}{v!}n^v p^l(1-p)^{v^*\frac{n^d}{d!}} \\
&\sim \frac{c}{v!}n^v p^l\exp\left(-p v^*\frac{n^d}{d!}\right) \\
&\sim\frac{c}{v!}n^v p^l\exp(-\log n-\omega\log\log n) \\
&\sim \frac{c}{v!}n^v\left(\frac{d!}{v^*}\log n\right)^ln^{-ld}n^{-1}(\log n)^{-\omega} \\
&\sim\frac{c}{v!}\left(\frac{d!}{v^*}\right)^l(\log n)^{l-\omega},
\end{align*}

 and the result follows. 
 
The ``in particular" part follows from the first moment method.

\end{proof}

Now consider the case $\omega\to l\in\mathbb{R}$ and let

            $$c(n):=p\frac{n^dv^*}{d!}-\log n-l\log\log n.$$

 \begin{Prop}
 
 Fix $\gamma\in\Gamma$ and $c=c(n)$ as above.
     
    \begin{enumerate}
    
         \item If $c\to-\infty$ then $\mathbb{E}[A(l,v^*,\gamma)]\to+\infty$
         \item If $c\to+\infty$ then $\mathbb{E}[A(l,v^*,\gamma)]\to 0$.

    \end{enumerate}
    
    In particular, the second condition implies that a.a.s. $A(l,v^*,\gamma)=0$.

 \end{Prop}
 
 \begin{proof}
 
 Note that $p\frac{n^dv^*}{d!}=\log n+l\log\log n+c(n)$, so 
 
 \begin{align*}
\mathbb{E}[A(l,v^*,\gamma)]&\sim \frac{c(l,v^*,\gamma)}{v!}n^v p^l(1-p)^{v^*\frac{n^d}{d!}} \\
&\sim \frac{c(l,v^*,\gamma)}{v!}n^v p^l\exp\left(-p v^*\frac{n^d}{d!}\right) \\
&\sim\frac{c(l,v^*,\gamma)}{v!}n^v p^l\exp(-\log n-l\log\log n-c(n)) \\
&\sim\frac{c(l,v^*,\gamma)}{v!}\left(\frac{d!}{v^*}\right)^l\exp(-c(n)),
\end{align*}

          and $1$ and $2$ follow.  
     
     The ``in particular" part follows from the first moment method.

 \end{proof}


Fix, in $G^{d+1}(n,p)$, a vertex set $S$ of size $|S|=1+ld$ and $\gamma\in\Gamma$. To each of the $c:=c(l,v^*,\gamma)$ potential copies of $v^*$-marked $l$-butterflies of type $\gamma$ in $S$ there corresponds
the random variable $X_\alpha$, the indicator of the event $B_\alpha$ that this potential copy is indeed there in $G^{d+1}(n,p)$. Then we clearly have

               $$A(l,v^*,\gamma)=\sum_\alpha X_\alpha.$$

We write $|X_\alpha|:=S$.

\begin{Prop} \label{intersectiontype}

Let $p\sim C\cdot\frac{\log n}{n^d}$, where $\frac{d!}{v^*+1}< C<\frac{d!}{v^*}$. Then, for any $k,l\in\mathbb{N}$, we have $$\mathbb{P}[A(l,v^*,\gamma)\geq k]\to 1.$$

\end{Prop}

\begin{proof}

By the first moment analysis, the condition on the hypothesis implies $\mathbb{E}[A(l,v^*,\gamma)]\to +\infty$ and $\mathbb{E}[A(\tilde{l},\tilde{v}^*,\gamma)]\to 0$ for all $\tilde{v}^*>v^*$ and any $\tilde{l}\in\mathbb{N}$. We use the second moment method. As

$$\mathbb{E}\left[\sum_{|X_\alpha|\cap |X_\beta|=\emptyset} X_\alpha X_\beta\right]\sim c^2\frac{n^{2v}}{v!^2}p^{2l}(1-p)^{2v^*\frac{n^d}{d!}}\sim\mathbb{E}[A(l,v^*,\gamma)]^2$$

it suffices to show that  

$$\mathbb{E}\left[\sum_{|X_\alpha|\cap |X_\beta|\ne\emptyset} X_\alpha X_\beta\right]=o(1).$$  

The sets $|X_\alpha|$ and $|X_\beta|$ can only  intersect according to a finite number of patterns,
so it suffices to show that the contribution of all terms with a given pattern is $o(1)$.
Set $S:=|X_\alpha|\cup |X_\beta|$.

Consider an intersection type $S$ such that the model spanned by $S$ contains a cycle. Then the 
configuration $S$ has a vertex of small degree (marked) near a cycle. The sum of contributions of all terms with that intersection type is $\sim$ the expected number of such configurations. As there
is a vertex of small degree near a cycle, this is $o(1)$ by proposition $35$.

If the type of $S$ has no cycles, then $S$ is a maked butterfly with $\tilde{v}^*\geq v^*$ marked vertices.

   If $\tilde{v}^*>v^*$ then, by the first moment analysis, the sum of contributions of those terms is 
   $o(1)$.

    We claim that there are no terms with $\tilde{v}^*=v^*$. Indeed, if that were the case,  by minimality, all edges would be in the intersection and so the events indicated by $X_\alpha$ and
    $X_\beta$ would be the same, a contradiction.

\end{proof}

Now consider $p\sim\frac{d!}{v^*}\cdot\frac{\log n}{n^d}$ and, as above, let 

$$\omega(n)=\frac{v^*n^d\frac{p}{d!}-\log n}{\log\log n}.$$

\begin{Prop}

Fix $\epsilon>0$ and $l\in\mathbb{N}$.

    \begin{enumerate}

      \item If eventually $\omega<l-\epsilon$ then for any $k\in\mathbb{N}$, we have $$\mathbb{P}[A(l,v^*,\gamma)\geq k]\to 1$$
      \item If $\omega\to+\infty$ then for any $k,\tilde{l}\in\mathbb{N}$, we have $$\mathbb{P}[A(\tilde{l},v^*-1,\tilde{\gamma})\geq k]\to 1$$ for all isomorphism types 
      $\tilde{\gamma}$ of minimal $(v^*-1)$-marked \\ $\tilde{l}$-butterfliies.
      
      \end{enumerate}

\end{Prop}

\begin{proof}

 The proof of $1$ is the same as the proof of the above proposition.
 
 The proof of $2$ is analogous, noting that condition $2$ implies that the expected number of
 $v^*$-marked butterflies with any fixed number of edges is $o(1)$, so that the intersection pattern must have all the $v^*-1$ marked vertices.

\end{proof}

Now consider the case $\omega\to l$ and, as above, let

     $$c(n)=\frac{pn^dv^*}{d!}-\log n-l\log\log n.$$
     
     Finally, the same reasoning used in the proofs of the two above propositions demonstrates the following.

\begin{Prop}

If $c\to-\infty$ then $\mathbb{P}[A(l,v^*,\gamma)\geq k]\to 1$ for any $k\in\mathbb{N}$.

\end{Prop}

              \subsection{Zero-one laws between the thresholds}

       Now we consider the countable models of the almost sure theories $T_p$
  for $p$ ``between" the critical values above.
  
  \begin{Thrm}
  
  Let $p$ be an edge function satisfying one of the following properties:
  
   \begin{enumerate}

                         \item $p\sim C\cdot\frac{\log n}{n^d}$, where $\frac{d!}{v^*+1}<C<\frac{d!}{v^*}$ for  some $d,v^*\in\mathbb{N}$.  
                         \item $p\sim \frac{d!}{v^*}\cdot\frac{\log n}{n^d}$ where $\omega\to\pm\infty$ or
            $\omega\to C$ where $l-1<C<l$ for some $l\in\mathbb{N}$.  
                                     
    \end{enumerate}    
    Then $p$ is a zero-one law.                 
  
  \end{Thrm}
  
  \begin{proof}
  
        Consider, first, a function $p\sim C\cdot\frac{\log n}{n^d}$, where
                       
                       $$\frac{d!}{v^*+1}<C<\frac{d!}{v^*}.$$
        We consider the models of the almost sure theory $T_{v^*}:=T_p$.             
       In that range, we still have no bicyclic (or more) components in the first-order perspective. As there are no vertices of 
       small degree near cycles, the unicyclic components are determined up to isomorphism.
       Also we still have infinitely many copies of each cycle. So the union of connected components 
       containing cycles are determined up to isomorphism and Duplicator does not have to
       worry about them: every time Spoiler plays there, he has wasted a move.
       
So let us consider the butterfly components. By the first and second moment analysis above,
we have no components containing $(v^*+1)$-marked butterflies and have infinite components containing copies  of each minimal $v^*$-marked butterfly. Each component containing a $v^*$-marked butterfly is determined
up to isomorphism: each non-marked vertex must have infinite neighbors.

Let $l\in\mathbb{N}$ such that $1+ld\le v^*<v^*+1\le 1+(l+1)d$. 
Then there are no butterflies of order $l+1$ (or more) as sub-hypergraphs and there are infinitely many
components isomorphic to each butterfly of order $\le l$.
Therefore, the union of the components isomorphic to finite butterflies is determined up to isomorphism.

$T_{v^*}$ is not $\aleph_0$-categorical, though, since in that countable models, there may or may
not be components containing $\tilde{v}^*$-marked butterflies with $\tilde{v}^*<v^*$. 
(This includes the degenerate case $\tilde{v}^*=0$: there may or may not be infinite butterflies 
where all vertices have infinite neighbors)
These 
components are ``simulated" by components containing $v^*$-marked vertices, with
$v^*-\tilde{v}^*$ marked vertices suitably far from the neighborhood of the $\tilde{v}^*$ marked vertices, this neighborhood being a copy of the $\tilde{v}^*$-marked butterfly one wants to
simulate.

More precisely, it is clear that the countable models of $T_p$ satisfy the hypothesis of 
proposition \ref{wining strategy}, so they are pairwise elementarily equivalent and, hence, $T_{v^*}$ is complete and the corresponding $p$ are 
zero-one laws.

Now consider $p\sim\frac{d!}{v^*}\cdot\frac{\log n}{n^d}$ and, as above, let 

$$\omega(n)=\frac{v^*n^d\frac{p}{d!}-\log n}{\log\log n}.$$ 

If $\omega\to-\infty$ then the first and second moment analysis above imply that the countable models of $T_p$ are the same as the countable models of $T_{v^*}$ and, as $T_{v^*}$ is complete, $p$ is a zero-one law.

If $\omega\to+\infty$ then the first and second moment analysis above imply that the countable models of 
$T_p$ are the same as the countable models of $T_{v^*-1}$. But the latter theory is complete and, hence, the corresponding $p$ are zero-one laws.

 If $\omega\to C$, with $l-1<C<l$, then the countable models of $T^l_{v^*}:=T_p$ are the same as the countable models of $T_{v^*}$ but without the components with marked butterflies of 
order $\le l-1$. These models are, for the same reasons, still pairwise elementarily equivalent,
so we have that the corresponding $p$ are zero-one laws.

Finally, consider the case $\omega\to l$ and, as above, let

                 $$c(n)=\frac{pn^dv^*}{d!}-\log n-l\log\log n.$$

     If $c\to-\infty$, then the analysis above show that the countable models of $T_p$ are the same
  as the countable models of $T^l_{v^*}$, so these $p$ are zero-one laws.
  
    If $c\to+\infty$, then the analysis above show that the countable models of $T_p$ are the same
  as the countable models of $T^{l-1}_{v^*}$, so these $p$ are also zero-one laws.
  
  \end{proof}

             \subsection{Axiomatizations}

    At this point, it is clear that the arguments given in the last section actually give axiomatizations
    for the almost sure theories presented there.
    
    More formally, let the theory $T(v^*)$ consist of a scheme of  axioms saying that there are no bicyclic (or more) components, a
    scheme of axioms saying that there are no copies of $\tilde{v}^*$-marked butterflies for each
    $\tilde{v}^*>v^*$ and a scheme of axioms saying that there are infinitely many copies of each minimal
    $v^*$-marked butterfly.  
    
    Similarly, let the theory $T(v^*,l)$ consist of a scheme of axioms saying that there are no bicyclic (or more)
    components, a scheme of axioms excluding the $\tilde{v}^*$-marked butterflies for each 
    $\tilde{v}^*>v^*$, a scheme of axioms saying that there are no $v^*$-marked butterflies of 
    order $\le l-1$ and an scheme saying that there are infinitely many copies of each minimal $v^*$-marked butterfly not excluded by the last scheme.      
    
    By the discussion found in the last section, we have the following:
    
    \begin{Thrm}
    
     The theory $T(v^*)$ is an axiomatization for $T_{v^*}$ and, similarly, the theory $T(v^*,l)$ is an axiomatization for
    $T^l_{v^*}$.
    
    \end{Thrm}

             \subsection{On the thresholds}

 The only way an $L$-function can avoid all of the clauses discussed above is the possibility that
 $c(n)$ converges to a real number $c$. That is to say, we must consider the possibility that
 
         $$p=\frac{d!}{v^*}\cdot\frac{\log n+l\log\log n+c(n)}{n^d}$$
         where $c(n)\to c$.       
             
 We will see, in the present chapter, that these $p$, although not zero-one laws, are still convergence laws. The situation is analogous to that of the last section:
 on these thresholds, the almost sure theories $T_p$ are almost complete.

    \subsubsection{Limiting Probabilities on the Thresholds}

Let $v^*,l\in\mathbb{N}$ and let $T_1,T_2,\ldots,T_u$ denote the collection of all possible (up to isomorphism) $v^*$-marked butterflies of order $l$ and let $I$ be the set of all $u$-tuples $\mathbf{m}=(m_1,\ldots,m_u)$ of non-negative integers.
Finally, let $\sigma_{\textbf{m}}$ be the elementary property that there are precisely $m_i$ components $T_i$ for $i\in\{1,\ldots,u\}$.

\begin{Prop}

Let  $p=\frac{d!}{v^*}\cdot\frac{\log n+l\log\log n+c(n)}{n^d}$, where $c(n)\to c$. Then the collection $\{\sigma_{\mathbf{m}}|\mathbf{m}\in I\}$ is a complete set of completions for $T_p$. In particular, $p$ is a convergence law.

\end{Prop}

\begin{proof}

We show properties $1,2,3$ and $4$ in the definition of a complete set of completions.

We claim the countable models of $T_p\cup\{\sigma_{\textbf{m}}\}$ are pairwise elementarily
equivalent. Indeed, the complement of the union of components containing the $v^*$-marked
butterflies of order $l$ is elementarily equivalent to the countable models of the theory 
$T_{v^*}^{l+1}$, defined above. As the latter theory is complete, $T_p$ is also complete, and
we have $1$.

Tautologically no two of the $\sigma_{\textbf{m}}$ can hold simultaneously, so we have property $2$.

For each $i\in\{1,2,\ldots,u\}$, let $\delta_i$ be the isomorphism type of $T_i$. For notational convenience, set $c_i:=c(l,v^*,\delta_i)$ and $A_i:=A(l,v^*,\delta_i)$. 
The next lemma implies properties $3$ and $4$ and, therefore, completes the proof.

\end{proof}

\begin{Lemma} \label{poisson2}

In the conditions of the above proposition, the random variables $A_1,A_2,\ldots,A_u$ are asymptotically independent Poisson with means 

            $$\lambda_i=\frac{c_i}{v!}\left(\frac{d!}{v^*}\right)^le^{-c}.$$ That is to say,

       $$p_{\textbf{m}}:=\lim_{n\to\infty}\mathbb{P}(\sigma_{\textbf{m}})=\prod^u_{i=1}e^{-\lambda_i}\frac{\lambda_i^{m_i}}{m_i!}.$$
In particular  $$\sum_{\textbf{m}\in I}p_{\textbf{m}}=1.$$

\end{Lemma}

\begin{proof}

By the method of factorial moments, is suffices to show that, for all $r_1,r_2,\ldots,r_u\in\mathbb{N}$ we have

                    $$\mathbb{E}\left[(A_1)_{r_1}\cdots(A_u)_{r_u}\right]\to\lambda^{r_1}\cdots\lambda^{r_u}.$$

As we have seen, each $A_i$ can be written as a sum of indicator random variables $A_i=\sum_{S,j} X^{i,j}_{S}$, each $X^{i,j}_S$ indicates the event $E_S^{i,j}$ that the $j$-th of the potential copies of 
 $v^*$-marked $l$-butterflies on the vertex set $S$ is present. Then
$$\mathbb{E}\left[(A_1)_{r_1}\cdots(A_u)_{r_u}\right]=
\sum_{S_1,\ldots,S_u,j_1,\ldots,j_u}\mathbb{P}[E_{S_1}^{1,j_1}\land\ldots
\land E^{u,j_u}_{S_u}].$$
The above sum splits into $\sum_1+\sum_2$ where $\sum_1$ consists of the terms with $S_1,\ldots,S_u$ pairwise disjoint. It is easy to see that if  

$$p=\frac{d!}{v^*}\frac{\log n+l\log\log n+c(n)}{n^d}$$ 
then
             $\sum_1\sim\prod_i\lambda^{r_i}$.
             
             Arguing as in proposition \ref{intersectiontype}, one sees that the contribution of each of the terms in
     $\sum_2$ with a given pattern of intersection is $o(1)$. Hence $\sum_2=o(1)$ and we are done.

\end{proof}

These pieces together prove the following.

\begin{Thrm}

All elements in $\bc$ are convergence laws.

\end{Thrm}

\begin{proof}

Just note that all $L$-functions on the above range must satisfy, with the familiar definitions of 
$\omega(n)$ and $c(n)$, one of the following conditions:

           \begin{enumerate}

                         \item $n^{-d}\ll p\ll(\log n)n^{-d}$
                         \item $(\log n)n^{-d}\ll p\ll n^{-d+\epsilon}$ for all positive $\epsilon$
                         \item $p\sim C\cdot\frac{\log n}{n^d}$, where $\frac{d!}{v^*+1}<C<\frac{d!}{v^*}$ for  some $d,v^*\in\mathbb{N}$
                .         \item $p\sim \frac{d!}{v^*}\cdot\frac{\log n}{n^d}$ where $\omega\to\pm\infty$ or
            $\omega\to C$ where $l-1<C<l$ for some $l\in\mathbb{N}$
                           \item $\omega\to l\in\mathbb{N}$ and $c(n)\to\pm\infty$ or $c(n)\to c\in\mathbb{R}$.
                                                    
            \end{enumerate}

\end{proof}

As it was the case in the last section, it is worth noting that the arguments used in getting zero-one laws for the clauses $1$, $2$ and $3$ do not require the edge functions to be in Hardy's class, so \emph{all} functions inside those intervals are zero-one laws, regardless of being $L$-functions.
 
 On the other hand, taking $\omega(n)$ oscillating infinitely often between two constant values 
 $\omega_1<l$ and $\omega_2>l$ makes the probability of an elementary event oscillate between zero and one. Similarly, taking $c(n)$ oscillating between any two
 different  positive values makes the probability of an elementary event oscillate between two different values $\notin\{0,1\}$. Obviously, these situations rule out convergence laws.

 As it was the case with $\bb$, our present discussion implies that, in a certain sense, most of the functions in $\bc$ are zero-one laws: the only way one of that functions can avoid this condition is being inside
one of the countable windows inside a threshold for the presence of marked butterflies of some order.

\section{Some elementary approximations}

In this chapter we describe some combinatorial aspects of the component structure of the random
hypergraph for $p\sim C\frac{\log n}{n^d}$ that will, later, be used to get elementary 
approximations to the non-elementary events $D_l$.
It all starts by showing that Erd\H os and R\'enyi's threshold for connectivity has a nice generalization for random hypergraphs.

\subsection{Threshold for Connectivity}

Now we show that $p=d!\frac{\log n}{n^d}$ is a threshold for $G^{d+1}(n,p)$ to be connected.
The $0$-statement we already know: if $p=C\cdot\frac{\log n}{n^d}$ with $C<d!$, then a.a.s. there
are many isolated vertices, so $G^{d+1}(n,p)$ is almost never connected.
The $1$-statement is the following.

\begin{Thrm} \label{conectivity}

Let $p(n)=\frac{C\log n}{n^d}$, where $C>d!$. Then a.a.s. the hypergraph $G^{d+1}(n,p)$ is connected.

\end{Thrm}

\begin{proof}

Let $a>0$ be such that $C>\frac{(d+1)!}{d+1-a^d}$. The expected number of cuts is less then or equal to

   $$\sum_{k=1}^{n/2}{n\choose k}(1-p)^{{n\choose d+1}-{n-k\choose d+1}-{k\choose d+1}}.$$ This sum is less then or equal to

   $$\sum_{k=1}^{an}{n\choose k}(1-p)^{{n\choose d+1}-{n-k\choose d+1}-{k\choose d+1}}+\sum_{k=an}^{n/2}{n\choose k}(1-p)^{{n\choose d+1}-{n-k\choose d+1}-{k\choose d+1}}$$
so it suffices to show that the two latter sums are $o(1)$.

The first sum is less than or equal to

    \begin{align*}
&\sum_{k=1}^{an}n^k\exp\left(-\frac{p}{(d+1)!}\left(n^{\underline {d+1}}-(n-k)^{\underline {d+1}}-k^{\underline {d+1}}\right)\right) \\
&\le\sum_{k=1}^{an}n^k\exp\left(-\frac{p}{(d+1)!}\left(k(d+1)n^d+o(n^d)-k^{\underline {d+1}}\right)\right) \\
&\le\sum_{k=1}^{an}n^k\exp\left(-\frac{pk}{(d+1)!}\left((d+1)n^d+o(n^d)-(k-1)^{\underline d}\right)\right) \\
&\le\sum_{k=1}^{an}n^k\exp\left(-\frac{pk}{(d+1)!}\left((d+1)n^d+o(n^d)-(an-1)^{\underline d}\right)\right) \\
&\le\sum_{k=1}^{an}n^k\exp\left(-\frac{Ck\log n}{(d+1)!n^d}\left((d+1)n^d+o(n^d)-(an-1)^{\underline d}\right)\right) \\
&\le\sum_{k=1}^{an}n^k\exp\left(-\frac{Ck\log n}{(d+1)!}\left((d+1)+o(1)-a^d)\right)\right) \\
&=\sum_{k=1}^{an}\exp\left(k\log n-\frac{Ck\log n}{(d+1)!}\left((d+1)+o(1)-a^d)\right)\right) \\
&=\sum_{k=1}^{an}\exp\left(k\log n\left(1-\frac{C}{(d+1)!}\left((d+1)+o(1)-a^d\right)\right)\right).
       \end{align*}

By the choice of $a$, there is $M>0$ such that, for sufficiently large $n$, the latter sum is less than or equal to

$$\sum_{k=1}^{an}\exp\left(-Mk\log n\right).$$ But we have

       $$\sum_{k=1}^{an}\exp\left(-Mk\log n\right)\leq\sum_{k=1}^{an}n^{-Mk}\leq\frac{n^{-M}}{1-n^{-M}}=o(1).$$

The second sum is less than or equal to

   $$\sum_{k=an}^{n/2}2^n\exp\left[-p\left({n\choose{d+1}}-{n-k\choose{d+1}}-{k\choose{d+1}}\right)\right].$$

The $\log$ of the summand is

        \begin{align*}
&n\log 2-\frac{C\log n}{(d+1)!n^d}\left[n^{\underline{d+1}}-(n-k)^{\underline {d+1}}- k^{\underline {d+1}}\right] \\
&\le n\log 2-\frac{Cn\log n}{(d+1)n^d}\left[(n-1)^{\underline d}-\left(1-\frac{k}{n}\right)(n-k-1)^{\underline d}- \frac{k}{n}(k-1)^{\underline d}\right] \\
&\le n\log 2-\frac{Cn\log n}{(d+1)n^d}\left[(n-1)^{\underline d}-\left(1-\frac{k}{n}\right)(n-1)^{\underline d}-\frac{k}{n}(n/2-1)^{\underline d}\right] \\
&\le n\log 2-\frac{Cn\log n}{(d+1)n^d}\left[\frac{k}{n}\left(n^d+o(n^d)\right)-\frac{k}{n}\left(\frac{n^d}{2^d}+o(n^d)\right)\right] \\
&\le n\log 2-\frac{Cn\log n}{(d+1)}\left[\frac{k}{n}\left(1-\frac{1}{2^d}\right)+o(1)\right] \\
&\le n\log 2-\frac{Cn\log n}{(d+1)}\left[a\left(1-\frac{1}{2^d}\right)+o(1)\right]
          \end{align*}

   This is the $\log$ of an individual summand. We have at most $n$ terms, so the sum is less then or equal to

    $$\exp\left[\log n+n\log 2-\frac{Cn\log n}{(d+1)}\left[a\left(1-\frac{1}{2^d}\right)+o(1)\right]\right]$$
and this is $o(1)$ as the $n\log n$ term overwhelms all others.

\end{proof}

\subsection{Component Structure on $\bc$}

It is possible to explore the information we have until now and, particularly, the arguments in the
last section, to get more precise information on the component structure of $G^{d+1}(n,p)$ for
$p\in\bc$.

\begin{Lemma}

Let $E_1$ be the event that there is a sub-hypergraph isomorphic to a butterfly with $l+1$ edges and $E_2$ the event that there is no connected component isomorphic to a butterfly on $l$ edges.

Then, for $p=p(n) \leq \frac{C\log n}{n^{d}}$, with $C<\frac{d!}{1+ld}$, the probability of the 
event $E_1\land E_2$ is $o(1)$.

\end{Lemma}

\begin{proof}

Fix a positive real $\alpha$ such that $\frac{1+(l+1)d}{l+1}<\alpha<\frac{1+ld}{l}$ and consider two cases:

If $p\leq n^{-\alpha}$ then a.a.s. there is no sub-hypergraph isomorphic to a butterfly with $l+1$ vertices.

If $p\geq n^{-\alpha}$ then a.a.s. there is a connected component isomorphic to a butterfly with $l$ vertices.

\end{proof}

\begin{Def}
 The random variable $\mu$ is the number of vertices outside the largest connected component of $G^{d+1}(n,p)$.

\end{Def}

The last part of the argument in theorem \ref{conectivity} gives the following lemma.

\begin{Lemma} \label{cuts}

Let $a>0$ and $C>0$. Then for $p(n)=C\frac{\log n}{n^d}$ one has a.a.s. $\mu<an$.

\end{Lemma}

\begin{proof}

This argument occurs, with virtually no changes, in the last part of the proof of theorem √∑\ref{conectivity}, but we 
repeat it here for the convenience of the reader.

If $\mu\geq an$ then $G^{d+1}(n,p)$ has a cut with at least $an$ vertices on the smaller side. The expected number of such cuts is less then or equal to

    $$\sum_{k=an}^{n/2} {n\choose k}(1-p)^{{n\choose{d+1}}-{n-k\choose{d+1}}-{k\choose{d+1}}}$$ which in turn is less than or equal to

   $$\sum_{k=an}^{n/2}2^n\exp\left[-p\left( {n\choose{d+1}}-{n-k\choose{d+1}}-{k\choose{d+1}}\right)\right].$$

The $\log$ of the summand is

        \begin{align*}
&n\log 2-\frac{C\log n}{(d+1)!n^d}\left[n^{\underline{d+1}}-(n-k)^{\underline {d+1}}- k^{\underline {d+1}}\right] \\
&\le n\log 2-\frac{Cn\log n}{(d+1)n^d}\left[(n-1)^{\underline d}-\left(1-\frac{k}{n}\right)(n-k-1)^{\underline d}- \frac{k}{n}(k-1)^{\underline d}\right] \\
&\le n\log 2-\frac{Cn\log n}{(d+1)n^d}\left[(n-1)^{\underline d}-\left(1-\frac{k}{n}\right)(n-1)^{\underline d}-\frac{k}{n}(n/2-1)^{\underline d}\right] \\
&\le n\log 2-\frac{Cn\log n}{(d+1)n^d}\left[\frac{k}{n}\left(n^d+o(n^d)\right)-\frac{k}{n}\left(\frac{n^d}{2^d}+o(n^d)\right)\right] \\
&\le n\log 2-\frac{Cn\log n}{(d+1)}\left[\frac{k}{n}\left(1-\frac{1}{2^d}\right)+o(1)\right] \\
&\le n\log 2-\frac{Cn\log n}{(d+1)}\left[a\left(1-\frac{1}{2^d}\right)+o(1)\right]
          \end{align*}

    This is the $\log$ of an individual summand. We have at most $n$ terms, so the sum is less then or equal to

    $$\exp\left[\log n+n\log 2-\frac{Cn\log n}{(d+1)}\left(a\left(1-\frac{1}{2^d}\right)+o(1)\right)\right]$$
and this is $o(1)$ as the $n\log n$ term overwhelms all others.

\end{proof}

\begin{Thrm}

Let $p\sim\frac{C\log n}{n^{d}}$ where $C$ is a constant such that $C>\frac{d!}{1+ld}$.
Then a.a.s. the random hypergraph $G^{d+1}(n,p)$ has no sub-hypergraph isomorphic to a butterfly 
on $l$ vertices outside the largest connected component.

\end{Thrm}

\begin{proof}

Fix positive reals $\tilde{C}>C$, $\alpha<\frac{d}{1+(l+1)d}$ and $a$ such that $a\tilde{C}<\frac{d!}{1+ld}$.

Let $A$ be the event that there is a sub-hypergraph isomorphic to a butterfly with $l+1$ edges outside the giant component and $B$ the event that there is a connected component isomorphic 
to a butterfly on $l$ vertices.

 For $p$ in this range, we know that $\mathbb{P}(B)=o(1)$. Therefore, it suffices to show that
 
                 $$\mathbb{P}(A)=o(1).$$

Let $\mu$ be the number of vertices outside the giant component and define $m_1=m_1(n)$ to be the natural number $m\in [0,n^{\alpha}]$ that maximizes 
                  $$\mathbb{P}(A|\mu=m).$$
             Similarly, define $m_2=m_2(n)$ to be the natural number $m\in [n^{\alpha},an]$ that maximizes $\mathbb{P}(A\land\lnot B|\mu=m)$.  One has
                      $$\mathbb{P}(A)=\mathbb{P}(A\land B)+\mathbb{P}(A\land \lnot B)\leq\mathbb{P}(B)+\mathbb{P}(A\land\lnot B).$$
As $\mathbb{P}(B)=o(1)$ it suffices to show that $\mathbb{P}(A\land\lnot B)=o(1)$. We have that
$\mathbb{P}(A\land\lnot B)$ is at most
$$\mathbb{P}(A\land\lnot B\land\mu\leq n^{\alpha})+\mathbb{P}(A\land\lnot B\land n^{\alpha}\leq\mu\leq an)+\mathbb{P}(\mu\geq an).$$

We show that the three latter terms are $o(1)$. 

The third term is $o(1)$ by lemma √∑\ref{cuts}.

For the first one, note that

\begin{align*}
\mathbb{P}(A\land\lnot B\land\mu\leq n^{\alpha})&\leq\mathbb{P}(A\land\mu\leq n^{\alpha}) \\
&=\sum_{m=0}^{n^{\alpha}} \mathbb{P}(\mu=m)\mathbb{P}(A|\mu=m) \\
&\leq\mathbb{P}(A|\mu=m_1).
\end{align*}

Note that $\mathbb{P}(A|\mu=m_1)$ is the probability that the random hypergraph on $m_1$ vertices and probability $p(n)$ has a sub-hypergraph isomorphic to a butterfly with $l+1$ edges, and that this is at most the expected number of butterflies with $l+1$ edges.
So 
\begin{align*}
\mathbb{P}(A|\mu=m_1)&=O\left[(\log n)n^{-d}m^{1+(l+1)d}_1\right] \\ &=O\left[(\log n)n^{\alpha(1+(l+1)d)-d}\right]=o(1)
\end{align*}
 by the choice of $\alpha$. 

As for the second term, one has

\begin{align*}
\mathbb{P}(A\land\lnot B\land\mu\geq n^{\alpha})&=
\sum_{n^{\alpha}}^{an}\mathbb{P}(\mu=m)\mathbb{P}(A\land\lnot B|\mu=m) \\
&\leq\mathbb{P}(A\land\lnot B|\mu=m_2).
\end{align*}

Note that $\mathbb{P}(A\land\lnot B|\mu=m_2)$ is the probability that the random hypergraph on $m_2$ vertices and probability $p(n)$ has a sub-hypergraph isomorphic to a butterfly on $l+1$ vertices and no connected component isomorphic to a butterfly on $l$ vertices. 

It is easy to see that the function $n\mapsto\frac{\log n}{n^{d-1}}$ is eventually decreasing, so that $\frac{\log n}{n^{d-1}}\le\frac{\log m_2}{m_2^{d-1}}$ for sufficiently large $n$. Moreover, it is obvious that $m_2<an$.

Putting all together, one has, for sufficiently large $n$,
$$p< \frac{\tilde{C}\log n}{n^d}\le
a\tilde{C}\cdot\frac{\log m_2}{m_2^d}<
\frac{d!}{1+ld}\cdot\frac{\log m_2}{m_2^d}$$
so that, by Lemma $52$ and the fact that $m_2\to\infty$, we have $$\mathbb{P}(A\land\lnot B|\mu=m_2)=o(1).$$

\end{proof}

For the reader who feels uneasy about the logic of the above argument, note, for example, that what is 
actually been done in the last part is the construction of a function 

             $$m_2\mapsto p(m_2)\in[0,1]$$
 beginning with the image of the funtion $m_2(n)$ and completing with values of $p$ that do not
contradict the already existing inequalities. The only reason for which the sequence

                        $$\mathbb{P}(A\land\lnot B|\mu=m_2(n))$$     
  fails to be a subsequence of 
                   $$\mathbb{P}\left[G^{d+1}(m_2,p(m_2))\models E_1\land E_2\right]$$
  is the fact that it may have repetitions. But, still, it is a \emph{sub-net} of the latter sequence,
  because $m_2(n)\to\infty$. As the latter converges to zero, this suffices to getting the desired
  conclusion $\mathbb{P}(A\land\lnot B|\mu=m_2(n))=o(1)$. Similar reasonings are needed
  in the arguments found in the next section but, due to their cumbersome but trivial nature, we will
  no longer bother the reader with such explicit formulations.

Putting all we already know about the existence of butterflies as components gives the following 
theorem, which gives a description of the ``disappearance" of the butterfly components as
time goes forth in the window $p\sim C\frac{\log n}{n^d}$. The butterflies of larger order are incorporated to the giant component before the
butterflies of smaller order, so, outside the giant component, what we see is the behavior of
$G^{d+1}(n,p)$ in BB but with time flowing backwards. 

In all that follows, let 
        $$c_l(n)=\frac{n^d(1+ld)}{d!}p(n)-\log n-l\log\log n.$$
        
 Notice that, if $\omega\to l$, then $c_l(n)$ is the usual $c(n)$.

\begin{Thrm}

Let $p$ be such that 
            $$\lim_{n\to\infty} c_l(n)=-\infty , \quad \lim_{n\to\infty} c_{l+1}(n)=+\infty.$$

Fix $k\in\mathbb{N}$.

Then a.a.s. the complement of the largest component  of $G^{d+1}(n,p)$ consists of
a disjoint union of butterflies of order at most $l$ and nothing else. Moreover, for each 
isomorphism type of each butterfly of order $\le l$, there are at least $k$ copies of butterflies
of that type as components. 

\end{Thrm}

\subsubsection{Some Estimates for $\mu$}

The arguments in the proof of the above theorem can be used to get upper and lower bounds
for $\mu$ that are much better than that of lemma \ref{cuts}.

\begin{Thrm}

Let $p\sim C\frac{\log n}{n^d}$ satisfy
        $$\lim_{n\to\infty} c_l(n)=-\infty.$$

Fix any function $f(n)\ll n^{\frac{d}{1+ld}}(\log n)^{-1-ld}$.

Then a.a.s. $\mu>f(n)$.

\end{Thrm}

\begin{proof}

Let $A$ be the event that there is a sub-hypergraph isomorphic to a butterfly of order $l$.
Then the above condition on $p$ implies $\mathbb{P}(\lnot A)=o(1)$.
We have 

\begin{align*}
\mathbb{P}(\mu\le f(n))&=\mathbb{P}(A\land\mu\le f(n))+\mathbb{P}(\lnot A\land\mu\le f(n)) \\
&\le\mathbb{P}(A\land\mu\le f(n))+\mathbb{P}(\lnot A).
\end{align*}

As $\mathbb{P}(\lnot A)=o(1)$, it suffices to show that $\mathbb{P}(A\land\mu\le f(n))=o(1)$.
To this end, let $m_1=m_1(n)$ be the natural number $m\in[0,f(n)]$ that maximizes
$\mathbb{P}(A\mid\mu=m)$. Then

$$\mathbb{P}(A\land \mu\le f(n))=\sum_{m=0}^{f(n)}\mathbb{P}(\mu=m)\mathbb{P}(A|\mu=m)\le\mathbb{P}(A|\mu=m_1).$$
Note that $\mathbb{P}(A\mid\mu=m_1)$ is the probability that the random hypergraph on
$m_1$ vertices and edge probability $p(n)$ has a sub-hypergraph isomorphic to a butterfly on
$l$ vertices and that this is at most the expected number of such butterflies. Then

\begin{align*}
\mathbb{P}(A|\mu=m_1)&=O\left[(\log n)n^{-d}m^{1+ld}_1\right] \\ &=O\left[(\log n)n^{-d}(f(n))^{1+ld}\right]=o(1)
\end{align*}
 by the condition on $f$.

\end{proof}

\begin{Thrm}

Let $p$ satisfy
       $$\lim_{n\to\infty} c_l(n)=+\infty.$$

Fix any function $f(n)\gg n^{\frac{ld}{1+ld}}(\log n)^{-\frac{l}{1+ld}}$.

Then a.a.s. $\mu<f(n)$.

\end{Thrm}

\begin{proof}

Let $A$ be the event that there is a connected component isomorphic to a butterfly of order $l$.
Then any of the above conditions on $p$ implies that $\mathbb{P}(A)=o(1)$.
Therefore, as above, it suffices to show that 

                         $$\mathbb{P}(\lnot A\land\mu\ge f(n))=o(1).$$
                         
To this end, fix $a$ and let $m_2$ be the natural number $m\in[f(n),an]$ that maximizes
$\mathbb{P}(\lnot A\mid\mu=m)$.  Then

\begin{align*}
\mathbb{P}(\lnot A\land \mu\ge f(n))&=\sum_{m=f(n)}^{an}\mathbb{P}(\mu=m)\mathbb{P}(\lnot A\mid\mu=m) \\ &\le\mathbb{P}(\lnot A\mid\mu=m_2).
\end{align*}

Note that $\mathbb{P}(\lnot A\mid\mu=m_2)$ is the probability that the random hypergraph on
$m_2$ vertices and edge probability $p(n)$ has no component isomorphic to a butterfly on $l$
vertices. But the conditions on $f$ and $a$ imply, for sufficiently large n,

                     $$m^{-\frac{1+ld}{l}}_2\ll p(n)\le\frac{d!}{1+ld}\cdot\frac{\log m_2}{m_2^d}.$$
                     Therefore $\mathbb{P}(\lnot A\mid\mu=m_2)=o(1)$.

\end{proof}

            \subsubsection{Some Elementary Approximations}

 The description of the structure outside the giant component given in the above
 section enables us to get good elementary approximations to the events $D_l$.
 Below we define a class of properties that are, in a sense, asymptotically very improbable.

 \begin{Def}
 
 $\mathcal{I}$ is the set of all properties $P$ of $(d+1)$-uniform hypergraphs such that
 $\mathbb{P}(P)\to0$ for all $L$-functions $p:\mathbb{N}\to[0,1]$.

 \end{Def}           
            
Note that $\mathcal{I}$ is an \emph{ideal}, in that if $P_1\subseteq P_2\in\mathcal{I}$ then
$P_1\in\mathcal{I}$. Also, $\mathcal{I}$ is non-trivial, since the tautological event $\top$ is not
an element of $\mathcal{I}$.      

We say two events $P_1$ and $P_2$ are \emph{asymptotically equivalent} if their symmetric
difference $P_1\triangle P_2$ is an element of $\mathcal{I}$. In that case we write

               $$P_1\equiv P_2\mod\mathcal{I}.$$
               
Let $A$ be the property that there are no components isomorphic to butterflies of order $l$ and
$B$ the property that the largest component is a butterfly of order $l$ and all other components
are butterflies of order $< l$.
 Then $\tilde{D}_l:=A\lor B$ is an elementary property and is a good approximation to the event
 $D_l$.

            \begin{Thrm}

For all $l\in\mathbb{N}$, $D_l\equiv\tilde{D}_l\mod\mathcal{I}$.

            \end{Thrm}

\begin{proof}

Note that $D_l\subseteq\tilde{D}_l$, so that $D_l\triangle\tilde{D}_l=\tilde{D}_l\setminus D_l$.

First, if $p\ll n^{-\frac{1+ld}{l}}$, then  $\lim_{n\to\infty}\mathbb{P}(D_l)=\lim_{n\to\infty}\mathbb{P}(\tilde{D}_l)=1$.

So suppose, from now on, that $p\gg  n^{-\frac{1+ld}{l}}$.

If $p$ satisfies
$$\lim_{n\to\infty}c_l(n)=+\infty$$

then $\lim_{n\to\infty}\mathbb{P}(D_l)=\lim_{n\to\infty}\mathbb{P}(\tilde{D}_l)=1$.

Now suppose $p\gg n^{-\frac{1+ld}{l}}$ satisfy
        $$\lim_{n\to\infty}c_l(n)=-\infty.$$

Then $\lim_{n\to\infty}\mathbb{P}(D_l)=\lim_{n\to\infty}\mathbb{P}(\tilde{D}_l)=0$.

There are two remaining cases:

\begin{enumerate}

       \item $p\sim Cn^{-\frac{1+ld}{l}}$
       \item $p=\frac{d!}{1+ld}\cdot\frac{\log n+l\log\log n+c(n)}{n^d}$, where $c(n)\to c\in\mathbb{R}$.

\end{enumerate}
In both of them, the almost sure properties we already know to hold in $G^{d+1}(n,p)$ easily 
imply that $\mathbb{P}(\tilde{D}_l\setminus D_l)\to 0$.

\end{proof}

   Note that in the non-trivial cases

\begin{enumerate}

       \item $p\sim Cn^{-\frac{1+ld}{l}}$
       \item $p=\frac{d!}{1+ld}\cdot\frac{\log n+l\log\log n+c(n)}{n^d}$, $c(n)\to c\in\mathbb{R}$

\end{enumerate}
Lemmas \ref{poisson1} and \ref{poisson2} imply, with the notation we find there, that

        $$\lim_{n\to\infty}\mathbb{P}[D_l]=\exp(-\sum^{u}_{i=0}\lambda_i)$$
 where $v=1+ld$ and $\lambda_i=\frac{c_i}{v!}C^l$ in the first case and $\lambda_i=\frac{c_id!}{v!v}e^{-c}$,
  in the second, which is a generalization of Erd\H os beautiful  ``double exponential"
 formula 
 
                     $$\lim_{n\to\infty}\mathbb{P}[G(n,p) \text{ is connected}]=e^{-e^{-c}}$$
for $p=\frac{\log n+c}{n}$.

    The above approximations are global, in the sense that they work for all ranges of $L$-functions $p$. There are other situations in which the approximations have a more local character, meaning
    that they work for some specific ranges of $p$.
    Consider, for example, the non-elementary predicate $C(x)$, meaning that $x$ belongs to the largest connected component of $G^{d+1}(n,p)$. Then the above discussions imply that if
    $p$ is appropriately large then $C(x)$ is almost surely equivalent to the predicate $\tilde{C}_l(x)$, meaning that $x$ does not belong to a component of order $\le l$. Obviously, $\tilde{C}_l(x)$ is
    elementary for all $l\in\mathbb{N}$.
    Formally,
    
    \begin{Prop}
    
    If $p$ satisfies
    $$\lim_{n\to\infty}c_l(n)=+\infty$$
then the predicate $C(x)$ is almost surely equivalent to the predicate $\tilde{C}_l(x)$.   
    
    \end{Prop}

                               \section{Future Directions}

    This work has considered zero-one and convergence laws for edge functions $p$ on the ranges 
    $p,1-p\gg n^{-\epsilon}$ for all $\epsilon>0$ and $p\ll n^{-d+\epsilon}$ for all $\epsilon>0$.
    There are many edge functions outside those ranges and it would be interesting to study
    the almost sure theories of some of them so as to get information about possible convergence 
    laws. For example, Shelah and Spencer, in \cite{shelah}, showed that if $\alpha\in(0,1)$ is an irrational
    number then $p=n^{-\alpha}$ is a zero-one law for the binomial random graph $G(n,p)$.
    So zero-one laws remain a frequent appearance outside the considered ranges, at least in the
    case $d=1$ of random graphs. It is natural to ask whether the functions $n^{-\alpha}$, for
    $\alpha\in(0,d)\setminus\mathbb{Q}$ are zero-one laws for $G^{d+1}(n,p)$. The methods
    introduced by Shelah and Spencer seem to apply, with minor modifications in this case, to
    give a positive answer to that question.
    
    The fact that all the above irrational powers of $n$ are zero-one laws could make one wonder
    whether it would be possible to improve the exponent $d$ on
              $$p\ll n^{-d+\epsilon}$$    
    while still getting convergence laws. Spencer, in \cite{spencer}, shows that if $\alpha$ is any rational
       number in $(0,1)$, then $n^{-\alpha}$ fails to be a convergence law, so the answer is 
       negative at least in the case $d=1$ of random graphs. Again, the methods Spencer introduces
       seem to apply to the other values of $d$ to give a negative answer, and, therefore, the
       exponent in the right hand of  $\bc$ is probably the best possible. Still, it would be interesting
       to have a complete description of the intervals of $L$-functions entirely made of convergence
       laws.
       
       One could also look for examples of other elementary approximations of non-elementary 
   properties, as it was the case with $D_l$ and $\tilde{D}_l$. An interesting more challenging project is to describe the asymptotic expressive
   power of the first order logic of uniform hypergraphs, that is, the class of all properties
   $P$ such that $P\equiv\tilde{P}\mod\mathcal{I}$ for some elementary property $\tilde{P}$.
   For example, if one could get a nice description of the elements of the ideal $\mathcal{I}$,
   then the asymptotic expressive powers of all classes of properties would also be described.

\vfill\eject
\providecommand{\bysame}{\leavevmode\hbox to3em{\hrulefill}\thinspace}
\providecommand{\MR}{\relax\ifhmode\unskip\space\fi MR }
\providecommand{\MRhref}[2]{%
  \href{http://www.ams.org/mathscinet-getitem?mr=#1}{#2}
}
\providecommand{\href}[2]{#2}


\begin{thebibliography}{10}

\bibitem{alon}
N. Alon and J. H. Spencer. \emph{The Probabilistic Method}.
Wiley-Interscience Series in Discrete Mathematics and Optimization.
Wiley, New Jersey, Third edition, Aug. 2008.


 \bibitem{bollobas}  B. Bollob\'as. \emph{Random Graphs}. Number 73 in Cambridge studies in
advanced mathematics. Cambridge University Press, Cambridge, UK,
Second edition, Jan. 2001

  \bibitem{bollobas2} B. Bollob\'as. \emph{Modern Graph Theory}. Graduate Texts in Mathematics.
Springer, New York, Aug. 2002. Corrected Edition.

 \bibitem{compton} K. J. Compton. \emph{0-1 laws in logic and combinatorics}. In R. I., editor,
Algorithms and Order, volume 255 of Advanced Study Institute Series C:
Mathematical and Physical Sciences, pages 353--383. Kluwer Academic
Publishers, Dordrecht, 1989.

\bibitem{ehr} A. Ehrenfeucht. \emph{An application of games to the completeness problem
for formalized theories}. Fundamenta Mathematicae, 49:129--141, 1961.  

 \bibitem{erdos} P. Erd\H os and A. R\'enyi. \emph{On random graphs, I}. Publicationes Mathematicae,
6:290--297, 1959.  

 \bibitem{erdos2} P. Erd\H os and A. R\'enyi. \emph{On the evolution of random graphs}. In Publication
of the Mathematical Institute of the Hungarian Academy of Sciences,
number 5 in Acta Math. Acad. Sci. Hungar., pages 17--61, 1960. 

  \bibitem{fagin} R. Fagin. \emph{Probabilities on finite models}. The Journal of Symbolic Logic,
41(1):50--58, Mar. 1976.  

  \bibitem{talanov} Y.V. Glebskii, D.I. Kogan, M.I. Liagonkii, V.A. Talanov, \emph{Range and degree of realizability of formulas the restricted predicate calculus}. Cybernetics 5: 142--154. 

  \bibitem{hardy} G. H. Hardy, \emph{Orders of Infinity}. Cambridge Tracts in Math. and Math. Phys. 12 (2nd edition), Cambridge, 1924.

\bibitem{lynch} J. F. Lynch. \emph{Probabilities of sentences about very sparse random graphs}.
Random Structures and Algorithms, 3(1):33--53, 1992.  

  \bibitem{pruzan} J.S. Pruzan and E. Shamir. \emph{Component structure in the evolution of random
  hypergraphs}. Combinatorica, 1985, Volume 5, Issue 1, pp 81--94.
  
 \bibitem{shelah} S. Shelah and J. Spencer. \emph{Zero-one laws for sparse random graphs}.
Journal of the American Mathematical Society, 1(1):97--115, Jan. 1988.  

 \bibitem{spencer2} J. Spencer. \emph{Zero-one laws with variable probability}. Journal of Symbolic
Logic, 58(1):1--14, 1993.  

 \bibitem{spencer} J. Spencer. \emph{The Strange Logic of Random Graphs}. Number 22 in
Algorithms and Combinatorics. Springer-Verlag, Berlin, 2001.  

 \bibitem{vants} A.G. Vantsyan. \emph{The evolution of random uniform hypergraphs}. In probabilistic problems in discrete mathematics, pages 126--131. Mosvov. Inst. Elektron.
  Mchinostroenya, 1987.   




\end{thebibliography}

\end{document}